\documentclass[11pt]{article}
\usepackage[pagebackref,colorlinks=true,pdfpagemode=none,urlcolor=blue,linkcolor=blue,citecolor=blue]{hyperref}

\usepackage{amsthm, amssymb, bm, bbm}
\usepackage{soul, color,cancel}
\usepackage[]{amsmath}
\usepackage[]{amsfonts}
\usepackage[]{fancyhdr}
\usepackage[]{graphicx}
\graphicspath{{/EPSF/}{../figures/}{./figures/}}

\newtheorem{theorem}{Theorem}[]
\newtheorem{lemma}[theorem]{Lemma}
\newtheorem{proposition}[theorem]{Proposition}

\newtheorem{remark}[theorem]{Remark}

\def \Cm {\mathbb{C}}

\def \Imm {\mathbb{I}}
\def \Km {\mathbb{K}}

\def \Rm {\mathbb{R}}

\def \Zm {\mathbb{Z}}

\def\D{\mathcal{D}}

\def\I{\mathcal{I}}

\def\O{\mathcal{O}}

\def\V{\mathcal{V}}

\newcommand{\Vol}{ \text{Vol }}

\newcommand{\where}{\quad\text{ where }}

\newcommand{\tr}{ {\text{tr }}}

\newcommand{\cout}[1]{}

\newcommand{\x}{\mathrm{x}}
\newcommand{\y}{\mathrm{y}}

\newcommand{\sgn}[1]{\,{\rm sign}(#1)}

\newcommand{\dprod}[2]{\left\langle{#1},{#2}\right\rangle}

\newcommand{\rotpi}{a}
\newcommand{\Aadj}{-A^*}

\title{The geodesic X-ray transform with a $GL(n,\mathbb{C})$-connection}
\author{Fran\c{c}ois Monard\thanks{Department of Mathematics, University of California, Santa Cruz, CA 95064. fmonard@ucsc.edu} \and Gabriel P. Paternain\thanks{Department of Pure Mathematics and Mathematical Statistics, University of Cambridge, Cambridge, CB3 0WB, UK. g.p.paternain@dpmms.cam.ac.uk}}

\date{}

\begin{document}
\maketitle

\begin{abstract}
    We derive reconstruction formulas for a family of geodesic ray transforms with connection, defined on simple Riemannian surfaces. Such formulas provide injectivity of all such transforms in a neighbourhood of constant curvature metrics and non-unitary connections with curvature close to zero. If certain Fredholm equations are injective in the absence of connection, then for any smooth enough connection multiplied by a complex parameter, the corresponding transform is injective for all values of that parameter outside a discrete set. Range characterizations are also provided.%, as well as numerical illustrations.    
\end{abstract}

%\tableofcontents

\section{Introduction}

This paper is concerned with the attenuated X-ray transform on a non-trapping surface. We shall consider attenuations determined by a $n\times n$ matrix $A$ of complex-valued 1-forms (a $GL(n,\mathbb{C})$-connection).

Consider $(M,g)$ a compact oriented Riemannian surface with smooth boundary. We let $SM = \{ (\x,v)\in TM:\ |v|_g = 1 \}$ be the unit tangent bundle with geodesic flow $\varphi_t:SM\to SM$, defined on the domain
\begin{align}
    \D := \{ (\x,v,t):\quad (\x,v)\in SM, \quad t\in [-\tau(\x,-v), \tau(\x,v)] \},
    \label{eq:D}
\end{align}
where $\tau(\x,v)$ is the first time at which the geodesic $\gamma_{(\x,v)}$ with initial conditions $(\x,v)$ hits the boundary $\partial M$.  Recall that $\varphi_{t}$ is defined as $\varphi_{t}(\x,v):=(\gamma_{(\x,v)}(t),\dot{\gamma}_{(\x,v)}(t))$, with infinitesimal generator $X_{(\x,v)} = \frac{d\varphi_t}{dt}(\x,v)|_{t=0}$.

The manifold is said to be {\it non-trapping} if $\tau(\x,v)<\infty$ for all $(\x,v)\in SM$. In this paper we consider non-trapping surfaces where $\partial M$ is strictly convex, meaning that the second fundamental form of $\partial M\subset M$ is positive definite.
This is already enough to imply that $M$ is a disk (cf. \cite[Proposition 2.4]{PSU13}). If in addition $(M,g)$ has no conjugate points we say that the surface is {\it simple}.

Given $A$, consider the matrix weight $w:SM\to GL(n,\mathbb{C})$ that arises as a solution
of the transport equation on $SM$:
\begin{equation}\label{transport 001}
Xw=wA,\quad w|_{\partial_{+}(SM)}=\mbox{\rm id},
\end{equation}
where $\partial_{+}(SM)$ denotes the set of $(\x,v)\in\partial (SM)$ such that $v$ points inside $M$, i.e. $\langle \nu(\x),v\rangle\leq 0$ where $\nu$ is the outer unit normal at $\partial M$.
We define the attenuated X-ray transform associated with the connection $A$, 
$$I_{A}:C^{\infty}(SM,\mathbb{C}^{n}) \to C^{\infty}(\partial_{+}(SM),\mathbb{C}^{n})$$
as:
\begin{equation}
I_{A}(f)(\x,v):=\int_{0}^{\tau(\x,v)}w(\varphi_{t}(\x,v))f(\varphi_{t}(\x,v))\,dt.
\label{eq:defI}
\end{equation}
If $f\in C^{\infty}(M,\mathbb{C}^{n})$ we shall set $I_{A,0}(f):=I_{A}(f\circ \pi)$ where $\pi:SM\to M$ is the footpoint projection.

A question of fundamental importance in the subject, is whether $I_{A,0}$ is injective.
In \cite{Paternain2012}, the authors prove that in the case of a unitary connection and a Higgs field, the corresponding ray transform on a simple surface is injective. In \cite{Paternain2013a}, the same authors provide a range characterization for the attenuated ray transform. 
The salient feature here is that the connection need not be {\em unitary} (or, equivalently, {\em Hermitian}) in the sense that $A \ne \Aadj$.

To put the X-ray transform \eqref{eq:defI} into perspective, consider first a general matrix weight $w:SM\to GL(n,\mathbb{C})$. For each fixed $\x\in M$, the quantity $w^{-1}Xw(\x,v)$ may be expanded in the velocities $v$ as
\begin{equation}
w^{-1}Xw(\x,v)=\Phi(\x)+A(\x,v)+\text{\rm higher\;order\;terms\;in\;} v.
\label{eq:expansion}
\end{equation}
Hence the transport equation \eqref{transport 001} tells that the X-ray transform with connection picks up precisely the term in the expansion above that gives linear dependence in velocities.
If we were to pick just $\Phi(\x)$ we would have a Higgs field or potential and for $n=1$ and $g$ flat, this reduces to the usual attenuated ray transform that is so prominent in SPECT (single photon emission computed tomography). Inversion formulas and range characterization for this very important case were obtained in \cite{ABK,Noa1,Noa2}. Even for $n=1$, if the weight $w$ is allowed to be arbitrary, the attenuated ray transform in 2D may not be injective \cite{Bo}, but one may speculate that if the expansion in \eqref{eq:expansion} is finite, injectivity may persist.
In dimensions $\geq 3$, the game changes and very general injectivity results have been obtained in \cite{Paternain2016}.

Besides the motivation coming from medical imaging and SPECT, there is another reason for considering the problem of injectivity of $I_{A,0}$ and it has to do with the {\it non-linear} inverse problem of recovering $A$ from its scattering data, or non-abelian X-ray transform. It would be impossible to do justice to the literature on the topic here, but we refer to \cite{E,No,Paternain2012, Paternain2016} and references therein; the last two references also have a discussion about a pseudo-linearization procedure that allows to connect the linear and non-linear problems. 

Our approach in order to invert such transforms explicitely is a generalization of inversion formulas derived in \cite{Pestov2004} and further analyzed in \cite{Krishnan2010} for geodesic ray transforms on simple surfaces. The first author then provided generalizations of such approaches in the case of symmetric differentials on simple surfaces \cite{Monard2013a} and recently provided inversion formulas for geodesic X-ray transforms with scalar Higgs-field type attenuations \cite{Monard2015} (that is, $w^{-1}Xw=\Phi(\x)$ and $n=1$).

In \cite{E} it is proved that $I_{A,0}$ is injective for an arbitrary $GL(n,\mathbb{C})$-connection when $(M,g)$ is a domain in $\mathbb{R}^{2}$; this uses a delicate theorem about existence of holomorphic integrating factors established in \cite[Theorem 5]{ER}. Generic injectivity for the case of simple manifolds, including the case when both $g$ and $A$ are real analytic is proved in \cite{Zhou2016}.
In \cite{PatSalo2017}, injectivity of $I_{A,0}$ is proved for an arbitrary $GL(n,\mathbb{C})$-connection whenever $(M,g)$ is a negatively curved simple manifold. In spite of all this progress the following question remains open:

\medskip

\noindent{\bf Question:} Let $(M,g)$ be a simple surface and $A$ a $GL(n,\mathbb{C})$-connection. Is $I_{A,0}$ injective?

\medskip
Note that the question has a positive answer for $n=1$, this follows essentially from the methods in \cite{SaloUhlmann2011}, cf. Proposition \ref{prop:transl2} below. Theorem \ref{thm:injectivity} below provides several new instances in which $I_{A,0}$ is proved to be injective, and we shall also provide range characterizations. Numerical simulations illustrating the effectiveness of our approach will appear in future work. We now proceed to state our results in detail.

\section{Main results} \label{sec:main}

Let $(M,g)$ a non-trapping Riemannian surface with strictly convex boundary and let $\nu$ denote the outer unit normal to $\partial M$.
The unit sphere bundle $SM$ is a $3$-dimensional compact manifold with boundary, which can be
written as the union $\partial(SM) =\partial
_{+}(SM) \cup
\partial _{-}(SM) $,
$$\partial _{\pm }(SM)
=\{(\x, v)\in \partial(SM) ,\;\mp \,\langle \nu(\x) ,v \rangle \geq 0\;\}.$$
The standard volume forms on $SM$ and $\partial(SM)$ that we will use are defined by 
$$
\begin{array}{rcl}
d\Sigma^{3}(\x,v) &=& dM_{\x}\wedge dS(v) \\
d\Sigma^{2}(\x,v) &=& d\partial M_{x}\wedge dS(v) \\
\end{array}
$$
where $dM$ (resp. $d\partial M$) is the volume form of $M$ (resp. $\partial M$), and $dS$ is the volume form of the unit circle $S_{\x}$ in $T_{\x}M$. 
For $(x,v)\in\partial_+(SM)$, let $\mu(x,v)=-\langle \nu(x), v \rangle$ and let $L^{2}_{\mu}(\partial_{+}(SM),\mathbb{C}^{n})$ be the space of $\mathbb{C}^{n}$-valued functions on $\partial_{+}(SM)$ with inner product
$$(u,v)_{L^{2}_{\mu}(\partial_{+}(SM),\mathbb{C}^{n})}=\displaystyle\int_{\partial_{+}(SM)}\langle u, v\rangle_{\mathbb{C}^{n}}\mu\,d\Sigma^{2}.$$

Suppose now that $A$ is a $GL(n,\mathbb{C})$-connection. This simply means that $A$ is an $n\times n$ matrix whose entries are 1-forms. Its curvature is defined as the 2-form $F_{A}:=dA+A\wedge A$ (i.e., $F_{A}$ is a matrix of 2-forms with components $(F_A)_{ij} = dA_{ij} + \sum_{k=1}^n A_{ik}\wedge A_{kj}$). Using the star operator $\star$ associated with the metric, we will often consider $\star F_{A}:M\to \Cm^{n\times n}$.

Let us define the attenuated ray transform with connection $A$, as follows. Let $X := \frac{d\varphi_t}{dt}|_{t=0}$ the geodesic vector field on the unit circle bundle, and let $X_\perp:= [X,V]$ where $V$ is the so-called {\em vertical derivative} (see Section \ref{sec:notation}). For $f\in C(SM, \Cm^n)$, we define $I_A f := u_A^f|_{\partial_+ (SM)}$, where $u_A^f$ denotes the unique solution $u$ to the transport problem
\begin{align*}
    Xu + Au = -f\quad (SM), \qquad u|_{\partial_- (SM)} = 0.
\end{align*}
In this case, $I_A:C(SM,\Cm^n)\to C(\partial_+(SM),\Cm^n)$ extends by density into a bounded operator $I_A:L^2(SM,\Cm^n)\to L^2_\mu(\partial_+(SM),\Cm^n)$. % and, if the boundary $\partial M$ is strictly convex, in the $L^2(SM,\Cm^n)\to L^2_\mu(\partial_+(SM),\Cm^n)$ functional setting as well.
Upon defining $\pi:SM\to M$ to be the canonical projection and $A_V:= V(A)$, restrictions of $I_A$ of interest are given by $I_{A,0} f := I_{A} [f\circ \pi]$ for $f\in C^\infty (M,\Cm^n)$, extendible to $L^2(M,\Cm^n)$ by continuity; $I_{A,\perp} f := I_A [(X_\perp - A_V) (f\circ \pi)]$ for $f\in C_0^\infty (M,\Cm^n)$ (i.e., a smooth function vanishing at the boundary), extendible to $H^1_0(M,\Cm^n)$ by continuity. Note that $I_{A,\perp}$ can also be defined on functions which do not vanish at $\partial M$, and the difference will be studied in much detail in Section \ref{sec:reduction}.

Upon defining the operators 
\begin{align*}
    W_A f &:= \pi_0 (X_\perp - A_V)u_A^f,\qquad f\in C^\infty(M,\Cm^n), \\
    W_{A,\perp} f &:= \pi_0 u_A^{(X_\perp -A_V)f}, \qquad f\in C^\infty_0(M,\Cm^n), 
\end{align*}
where $\pi_0:L^2(SM)\to L^2(M)$ denotes the {\em fiberwise average} $\pi_0 f(\x) := \frac{1}{2\pi} \int_{S_\x} f(\x,v)\ dS(v)$, we first derive the following formulas, true on any non-trapping Riemannian surface with strictly convex boundary:
\begin{theorem}\label{thm:Fred} Let $(M,g)$ be a non-trapping Riemannian surface with boundary. Then the following equations hold: 
    \begin{align}
	f + W_A^2 f &= \frac{1}{8\pi} I_{\Aadj,\perp}^* B_{A,+} H Q_{A,-} I_{A,0}f, \qquad f\in C^\infty (M,\Cm^n)    \label{eq:FredWA} \\
	f + W_{A,\perp}^2 f &= - \frac{1}{8\pi} I_{\Aadj, 0}^* B_{A,+} H Q_{A,-} I_{A,\perp}f, \qquad  f\in C^\infty_0(M,\Cm^n). \label{eq:FredWAstar}
    \end{align}       
\end{theorem}
Formulas \eqref{eq:FredWA}-\eqref{eq:FredWAstar} take the form of {\em filtered-backprojection} algorithms, where the operator $B_{A,+} H Q_{A,-}$ (defined in Section \ref{sec:boundaryop}, see \eqref{eq:defQ} and \eqref{eq:defB}) can be viewed as a filter in data space, while the operators $I_{\Aadj,0}^*, I_{\Aadj,\perp}^*$, formal adjoints of $I_{\Aadj,0}:L^2(M,\Cm^n)\to L^2_\mu (\partial_+ (SM),\Cm^n)$ and $I_{\Aadj,\perp}:H^1_0(M,\Cm^n)\to L^2_\mu (\partial_+ (SM),\Cm^n)$ respectively, are sometimes referred to as backprojection operators. 

\begin{remark} While the transform $I_{A,\perp}$ can be defined for smooth functions with non-zero boundary values, Equation \eqref{eq:FredWAstar} no longer holds in this augmented space, as is illustrated on the Euclidean transform $I_\perp$ wihout connection in \cite[Proposition 5]{Monard2017}. There, decomposing $f\in H^1(M)$ into $f = f_0 + f_\partial$ where $f_\partial$ is the harmonic extension of the trace of $f$ and $f_0 \in H^1_0(M)$, it is then shown that formula \eqref{eq:FredWAstar} applied to $I_\perp f$ recovers $f_0 + \frac{1}{4} f_\partial$ and not $f_0 + f_\partial$.
\end{remark}

If, in addition, $(M,g)$ is simple, then the operators $W_A$ and $W_{A,\perp}$ extend as compact operators $W_A,W_{A,\perp}:L^2(M,\Cm^n)\to L^2(M,\Cm^n)$, see Lemma \ref{lem:compact} below. In particular, equations \eqref{eq:FredWA} and \eqref{eq:FredWAstar} are Fredholm equations, invertible up to a finite-dimensional kernel. In fact, we can prove something with stronger implications: 
\begin{theorem}\label{thm:analytic}
    For any analytic $C^1(M, (\Lambda^1)^{n\times n})$-valued family of connections $\lambda\mapsto A_\lambda$, the corresponding $L^2(M,\Cm^n)\to L^2(M,\Cm^n)$-valued families of operators $\lambda\mapsto W_{A_\lambda}$ and $\lambda\mapsto W_{A_\lambda,\perp}$ are analytic.     
\end{theorem}

By analytic Fredholm theory (see, e.g., \cite[Thm. VI.14]{Reed1981}), Theorem \ref{thm:analytic} implies that for $\lambda\mapsto A_\lambda$ analytic $C^1(M, (\Lambda^1)^{n\times n})$-valued, if $Id + W_{A_{\lambda_0}}^2$ is invertible for some value $\lambda_0$, then this remains true for all complex values $\lambda$ outside a discrete set, which from the Fredholm equations implies that $I_{A_\lambda,0}$ is injective for all such values (note that if $I_{A,0}(f)=0$, then $f\in C_0^\infty (M,\Cm^n)$, cf. Proposition \ref{prop:jet} below). For similar purposes of generic uniqueness in inverse problems, prior uses of analytic Fredholm theory have appeared for instance in \cite{SU2008} in the case of the radiative transport equation, and in \cite{Sun1991} in the case of Calder\'{o}n's inverse conductivity problem. 

We then focus on obtaining estimates for the error operators $W_A, W_{A,\perp}$, whose study starts in \cite{Pestov2004,Krishnan2010} in the case without connection (call $W\equiv W_0$ the corresponding operator). Obtaining transparent estimates is not obvious, as the constants derived in \cite{Krishnan2010} are not well-controlled by intrinsic geometric quantities. Such estimates have recently been obtained in \cite{Guillarmou} on surfaces with negative curvature, allowing non-trivial trapping. In an attempt to quantify simplicity and obtain more transparent error estimates, we recall that the absence of conjugate points is equivalent to the non-vanishing of the following function $b:\D\to \Rm$ outside the set $\{(\x,v,0):\ (\x,v)\in SM\}$, solution of 
\[  \ddot b + \kappa(\gamma_{\x,v}(t)) b = 0, \qquad b(\x,v,0) = 0, \qquad \dot b (\x,v,0) = 1, \quad (\x,v,t)\in \D. \]
Since additionally, $\lim_{t\to 0^+} \frac{|b(\x,v,t)|}{t} = 1$ for every $(\x,v)\in SM$, and since $\D$ is compact, the following claim is obvious
\begin{align}
    \begin{split}
	\text{If } (M,g) & \text{ is simple, there exist positive constants } C_1(M,g) \text{ and } C_2(M,g) \\ 
	\text{ such that } & C_1 t \le |b(\x,v,t)| \le C_2 t \text{ for every } (\x,v,t)\in \D \text,
    \end{split}    
    \label{eq:simpleclaim}
\end{align}
in which case we say that $(M,g)$ is a simple Riemannian surface with constants $C_1, C_2$. A finer analysis of the Schwartz kernels of the error operators then allows to prove the following theorem. In the statement, for an $n\times n$ matrix $M$, we denote $\|M\| := (\tr (B^*B))^{\frac{1}{2}}$ its Frobenius norm.  

\begin{theorem} \label{thm:estimate} Let $(M,g)$ be a simple Riemannian surface with constants $C_1, C_2$ as in \eqref{eq:simpleclaim} and Gaussian curvature $\kappa(\x)$. Given the $C^1$ connection $A$ with curvature $F_A$, let us denote $\alpha_A := \sup_{(\x,v)\in SM}\{ \|(A + A^*)/2\| (\x,v)\}$ and $\tau_\infty$ the diameter of $M$. There exist constants $C,C'$ depending on $(n, C_1, C_2, \tau_\infty,\alpha_A)$ such that 
    \begin{align}
	\|W_A\|_{L^2\to L^2}, \|W_{A,\perp}\|_{L^2\to L^2} \le \left(\frac{\Vol M}{2\pi}\right)^{\frac{1}{2}} \sqrt{C \|\star F_A\|^2_\infty + C' \|d\kappa\|^2_\infty}.
	\label{eq:WAest}
    \end{align}
\end{theorem}

As consequences of Theorems \ref{thm:analytic} and \ref{thm:estimate}, we obtain the following main conclusions. 

\begin{theorem} \label{thm:injectivity} Let $(M,g)$ be a simple surface and $A$ a $C^1$ connection. Then the following conclusions hold: 
    \begin{itemize}
	\item[$(i)$] If $\kappa$ is constant and $A$ is flat, the operators $W_A$ and $W_{A,\perp}$ vanish identically and Theorem \ref{thm:Fred} implies that the transforms $I_{A,0}$, $I_{A,\perp}$, $I_{\Aadj, 0}$ and $I_{\Aadj, \perp}$ are all injective, with explicit, one-shot inversion formulas. 
	\item[$(ii)$] Injectivity still holds if $(n, C_1, C_2, \tau_\infty, \alpha_A,\|\star F_A\|_\infty,\|d\kappa\|_\infty, \Vol M)$ are such that the right hand side of \eqref{eq:WAest} is less than $1$, with a Neumann series type inversion.
	\item[$(iii)$] If $(M,g)$ is such that the operator $Id + W^2$ is injective, then for every $\lambda\in \Cm$ outside a discrete set, the transforms $I_{\lambda A,0}$, $I_{\lambda A,\perp}$, $I_{-\lambda A^*,0}$ and $I_{-\lambda A^*,\perp}$ are all injective. 
    \end{itemize}    
\end{theorem}

In the statement of Theorem \ref{thm:injectivity}, injectivity of $I_{\Aadj,0}$ and $I_{\Aadj,\perp}$ comes from the fact that one may consider the Fredholm equations \eqref{eq:FredWA}-\eqref{eq:FredWAstar} corresponding to $I_{\Aadj,0}$ and $I_{\Aadj,\perp}$, and since we prove in Lemma \ref{lem:Wops} that $W_A$ and $W_{\Aadj,\perp}$ are $L^2(M,\Cm^n)\to L^2(M,\Cm^n)$ adjoints, then invertibility of $Id + W_A^2$ is equivalent to invertibility of $Id + W_{\Aadj,\perp}^2$. It is conjectured that $Id+W^{2}$ is injective on any simple surface.

Finally, we provide a range characterization for the operators $I_{A,0}$ and $I_{A,\perp}$ whenever the operator $I_{-A^*,0}$ is injective. In order to obtain such a characterization, we must first establish a series of results building equivalence of injectivities between transforms with different connections. This takes us to formulating a few key results. In what follows, we denote by $I_{A,m} := I_A|_{\Omega_m}$, where $\Omega_{m}$ is defined as $\Omega_{m}=\text{\rm Ker}(V-im Id)\cap C^{\infty}(SM,\mathbb{C}^{n})$ and $V$ is the vertical vector field (see Section \ref{sec:prelim}).
\begin{itemize}
    \item If $I_{A,0}$ is injective, then so is $I_{A+\omega\Imm_n,0}$ for any scalar one-form $\omega$. See Proposition \ref{prop:transl2}.
    \item If $I_{A,0}$ is injective, then $I_{A,m}$ is injective for any $m\in \Zm$. See Proposition \ref{prop:shift}.
    \item $I_{A,0}$ is injective if and only if $I_{A,\perp}$ is injective. In particular, all conclusions above hold if we only assume $I_{A,\perp}$ injective instead. See Propositions \ref{prop:IperpI0} and \ref{prop:shift}. 
\end{itemize}

For the range characterization, we extend $I_{A,\perp}$ to all functions in $C^\infty(M,\Cm^n)$ and not just those vanishing at the boundary. One may define the formal operators $P_{A,\pm}:= B_{A,-} H_\pm Q_{A,+}$, where $B_{A,-}$ and $Q_{A,+}$ are defined in Section \ref{sec:boundaryop} and $H_\pm$ denote odd and even fiberwise Hilbert transforms (see Section \ref{sec:prelim}). The operators $P_{A,\pm}$, defined in the smooth setting on a space denoted ${\mathcal S}_A^{\infty}(\partial_{+}(SM),\Cm^n)$ (see \eqref{eq:Sinf}), are boundary operators which only depend on the scattering relation and the scattering data $C_A$, and they allow to describe the ranges of $I_{A,0}$ and $I_{A,\perp}$ as follows. 

\begin{theorem}[Range characterization of $I_{A,0}$ and $I_{A,\perp}$]\label{thm:rangeCharac}
    Suppose that $(M,g)$ is a simple surface and $I_{-A^*,0}$ is injective, and let $\I\in C^\infty(\partial_+ (SM),\Cm^n)$. Then the following claims hold. 
    \begin{itemize}
	\item[$(i)$] $\I$ belongs to the range of $I_{A,0}:C^\infty(M,\Cm^n)\to C^\infty(\partial_+ (SM),\Cm^n)$ if and only if there exists $w\in {\mathcal S}_A^{\infty}(\partial_{+}(SM),\Cm^n)$ such that $\I = P_{A,-}w$. 
	\item[$(ii)$] $\I$ belongs to the range of $I_{A,\perp}:C^\infty(M,\Cm^n)\to C^\infty(\partial_+ (SM),\Cm^n)$ if and only if there exists $w\in {\mathcal S}_A^{\infty}(\partial_{+}(SM),\Cm^n)$ such that $\I = P_{A,+}w$. 
    \end{itemize}
\end{theorem}

Such range characterizations were previously established in \cite{Pestov2004} in the case without connection, in \cite{Paternain2013a,AA2015} in the case of unitary connections and Higgs fields, and recently in the case of the attenuated transform \cite{Assylbekov2016}. The range characterization for $I_0$ was recently proved by the first author to be the generalization of the classical moment conditions for compactly supported functions in the Euclidean case, see \cite[Theorem 2.3]{Monard2016}.

\medskip
\noindent{\bf Outline.} The remainder of the article is organized as follows. We recall generalities on the geometry of the unit circle bundle, transport equations with connection, with additional remarks on the symmetries in the data space $L^2_\mu (\partial_+ (SM),\mathbb{C}^{n})$ in Section \ref{sec:prelim}. In Section \ref{sec:inversion}, we prove Theorem \ref{thm:Fred} and study the error operators $W_A,W_{A,\perp}$ in detail, including the proof of Theorem \ref{thm:estimate}. In Section \ref{sec:reduction}, injectivity of ray transforms corresponding to different connections or different harmonic levels are inter-related, and the relation between the transform $I_A$ over one-forms and the transform $I_{A,\perp}$ is refined. Finally, based on additional preparatory results from Section \ref{sec:reduction} (namely, Proposition \ref{prop:transl2}), Section \ref{sec:range} presents the range characterization and the proof of Theorem \ref{thm:rangeCharac}. %Finally, Section \ref{sec:numerics} covers numerical examples. 

%%%%%%%%%%%%%%%%%%%%%%%%%%%%%%%%%%%%%%%%%%%%%%%%%%%%%% PRELIMINARIES

\section{Preliminaries} \label{sec:prelim}

\subsection{Setting and notation} \label{sec:notation}

Throughout this section we will assume that $(M,g)$ is a non-trapping surface with strictly convex boundary. As a consequence, it is simply connected (hence orientable).

\medskip

\noindent{\bf Geometry of the unit tangent bundle.} We briefly recall standard notation for the unit sphere bundle $SM$, see e.g. \cite{Paternain2012} for more detail. The vector field $X\in T(SM)$ can be completed into a global framing $\{X,X_\perp,V\}$ of $T(SM)$ with structure equations
\begin{align}
    [X,V] = X_\perp, \qquad [X_\perp,V] = -X, \qquad [X,X_\perp] = -\kappa V, \qquad (\kappa:\text{Gaussian curvature}).
    \label{eq:structure}
\end{align}
The Sasaki metric on $T(SM)$ is then the unique metric making this frame orthonormal, with volume form which we denote $d\Sigma^3$. This measure gives rise to an inner product space $L^2(SM,\mathbb{C}^{n})$, where the circle action on tangent fibers induces the orthogonal decomposition 
\begin{align*}
    L^2(SM,\mathbb{C}^{n}) = \bigoplus_{k\in \Zm} H_k, \qquad H_k := \ker (V - ikId).
\end{align*}
Upon defining $\Omega_k = C^\infty(SM,\mathbb{C}^{n}) \cap H_k$, a function $u\in C^\infty(SM,\mathbb{C}^{n})$ decomposes uniquely as $u = \sum_{k\in \Zm} u_k$ where each $u_k$ belongs to $\Omega_k$. If $u\in L^2(SM,\mathbb{C}^{n})$, then each $u_k$ belongs to $H_k \cong q^k L^2(M,\mathbb{C}^{n})$, where $q$ denotes a non-vanishing element of $C^\infty(SM,\Cm) \cap \ker (V-iId)$, whose existence is guaranteed by simple connectedness. 

\medskip

\noindent{\bf Scattering relation.} For $(\x,v)\in SM$, let us denote 
\begin{align*}
    \varphi_-(\x,v) := \varphi_{-\tau(\x,-v)}(\x,v) \in \partial_+ (SM), \qquad \varphi_+ (\x,v) := \varphi_{\tau(\x,v)}(\x,v)\in \partial_- (SM)
\end{align*}
both endpoints of the geodesic passing through $(\x,v)$. Let $\alpha:\partial (SM)\to \partial (SM)$ the {\em scattering relation}, i.e. $\alpha|_{\partial_\pm (SM)} = \varphi_\pm|_{\partial_\pm (SM)}$. 

\medskip
\noindent{\bf Transport equations on the unit tangent bundle.} As in the Introduction, for $f:SM\to \Cm^n$ and $A$ a $GL(n,\mathbb{C})$-connection, we define $u_A^f$ to be the unique solution to the transport problem 
\begin{align*}
    Xu + Au = -f\quad (SM), \qquad u|_{\partial_- (SM)} = 0.
\end{align*}
Let us denote $U_A:SM\to GL(n,\Cm)$ the unique matrix solution $W$ to the problem 
\begin{align*}
    XW + AW = 0\quad (SM), \qquad W|_{\partial_+ (SM)} = \Imm_n.
\end{align*}
From this solution, we define the {\em scattering data} 
\begin{align}
    C_A:\partial_- (SM) \to GL(n,\Cm), \qquad C_A = U_A|_{\partial_- (SM)}.    
    \label{eq:CA}
\end{align}
We also define the attenuation function $E_A:\D\to GL(n,\Cm)$ as
\begin{align}
    E_A(\x,v,t) := U_A(\varphi_t(\x,v)) U_A^{-1} (\x,v), \qquad (\x,v,t) \in \D,    
    \label{eq:EA}
\end{align}
unique solution of the $(x,v)$-dependent ODE 
\begin{align*}
    \frac{d}{dt} E_A (\x,v,t) + A(\varphi_t(\x,v))E_A (\x,v,t) = 0, \qquad (\x,v,t)\in \D, \qquad E_A(0,\x,v) = \Imm_n,
\end{align*}
and in terms of which many kernels will be expressed below. For $h$ defined on $\partial_+ (SM)$, define $h_{\psi, A}$ the unique solution $u$ to the transport problem 
\begin{align*}
    Xu + Au = 0\quad (SM), \qquad u|_{\partial_+ (SM)} = h.
\end{align*}
With the definition of $U_A$, we have, for $(\x,v)\in SM$,
\begin{align*}
    u_A^f (\x, v) &= U_A(\x,v) \int_0^{\tau(\x,v)} U_A^{-1} (\varphi_t (\x,v)) f(\varphi_t(\x,v))\ dt = \int_0^{\tau(\x,v)} E_A^{-1} (\x,v,t) f(\varphi_t(\x,v))\ dt, \\
    h_{\psi,A}(\x,v) &= U_A(\x,v) h(\varphi_{-\tau(\x,-v)}(\x,v)) = U_A(\x,v) h(\varphi_-(\x,v)),
\end{align*}
or, equivalently for the last one,
\begin{align*}
    h_{\psi,A}(\varphi_t(\x,v)) = U_A(\varphi_t(\x,v)) h(\x,v), \qquad (\x,v)\in \partial_+(SM), \quad t\in [0,\tau(\x,v)]. 
\end{align*}

As the matrix $A$ is not necessarily skew-Hermitian, the connections $A$ and $\Aadj$ are distinct, though we will see below that it is helpful to consider the transforms associated to both jointly. The first important identity to notice is 
\begin{align}
    U_A^* = U_{\Aadj}^{-1} \quad\text{on}\quad  SM,
    \label{eq:UUstar}
\end{align}
since both functions coincide with the unique solution $W$ to the transport problem
\begin{align*}
    XW + W A^* = 0 \quad (SM), \qquad W|_{\partial_+ (SM)} = \Imm_n.
\end{align*}

\medskip

\noindent{\bf Decomposition of $X+A$ and the Guillemin-Kazhdan operators.}
We may decompose $X + A = \mu_+ + \mu_-$ where $\mu_\pm: \Omega_k \to \Omega_{k\pm 1}$ is defined by $\mu_\pm :=\eta_{\pm}^{A}:=\eta_\pm + A_{\pm 1}$, where $\eta_{\pm} := \frac{1}{2} (X \pm i X_\perp)$ are the Guillemin-Kazhdan operators, see \cite{Guillemin1980}. Then $\frac{1}{i} (\mu_+ - \mu_-) = X_\perp - A_V$, where $A_V := V(A) = i(A_1-A_{-1})$.
Moreover
\[\mu_{\pm}^{*}=(\eta_{\pm}^{A})^{*}=-\eta_{\mp}^{-A^{*}}.\]
 In this paper we work exclusively with the case in which $M$ is a disk, hence we can consider global isothermal coordinates $(x,y)$ on $M$ such that the metric can be written as $ds^2=e^{2\lambda}(dx^2+dy^2)$ where $\lambda$ is a smooth
real-valued function of $(x,y)$. This gives coordinates $(x,y,\theta)$ on $SM$ where
$\theta$ is the angle between a unit vector $v$ and $\partial/\partial x$. Then $\Omega_k$ consists of all functions $u=h(x,y) e^{ik\theta}$ where $h \in C^{\infty}(M,\mathbb{C}^n)$. In these coordinates, a connection $A = A_z dz + A_{\bar{z}} d\bar{z}$ (with $z = x+iy$) takes the form $A(x,y,\theta) = e^{-\lambda} (A_z(x,y) e^{i\theta} + A_{\bar{z}}(x,y) e^{-i\theta})$, and we can give an explicit description of the operators $\mu_{\pm}$ acting
on $\Omega_{k}$. For $\mu_{-}$ we have (cf. \cite[Equation (24)]{Paternain_survey}):
\begin{equation}
    \mu_{-}(u)=e^{-(1+k)\lambda}\left(\bar{\partial}(he^{k\lambda})+A_{\bar{z}}he^{k\lambda}\right)e^{i(k-1)\theta}, \qquad u= h(x,y) e^{ik\theta}.
\label{eq:mu}
\end{equation}
%where $A_{-1}$ and $A_{\bar{z}}$ are related by $A_{-1}=e^{-\lambda}A_{\bar{z}}e^{-i\theta}$.
From this expression we may derive the following lemma which will be used later on:

\begin{lemma} Given $f\in\Omega_{k-1}$, there is $u\in\Omega_{k}$ and $v\in \Omega_{k-2}$ such that $\mu_{-}u=f$
and $\mu_{+}v=f$.
\label{lemma:solveCR}
\end{lemma}

\begin{proof} We only prove the claim for $\mu_{-}$, the one for $\mu_{+}$ is proved similarly. If we write $f=ge^{i(k-1)\theta}$, using (\ref{eq:mu}) we see that we only need to find $h\in C^{\infty}(M,\mathbb{C}^n)$
such that
\begin{equation}
\bar{\partial}(he^{k\lambda})+A_{\bar{z}}he^{k\lambda}=e^{(1+k)\lambda}g.
\label{eq:preCR}
\end{equation}
But it is well known that there exists a smooth $F:M\to GL(n,\mathbb{C})$ such that $\bar{\partial}F+A_{\bar{z}}F=0$, hence the solvability
of (\ref{eq:preCR}) reduces immediately to the standard solvability result for the Cauchy-Riemann operator, namely, given a smooth $b$, there is $a$ such that $\bar{\partial}a=b$. The existence of $F$ above follows right away from the fact that a holomorphic vector bundle over the disk is holomorphically trivial \cite[Theorems 30.1 and 30.4]{Foster91}, see also \cite{ER,NakUhlmann2002} for alternative proofs. 
\end{proof}

\medskip

\noindent{\bf Hilbert transform and commutator formulas.} An important operator for what follows is the {\em fiberwise Hilbert transform} $H:L^2(SM,\mathbb{C}^{n})\to L^2(SM,\mathbb{C}^{n})$, diagonal on the harmonic decomposition in the fiber, and such that $H|_{H_k} = -i\sgn{k} Id|_{H_k}$, with the convention $\sgn{0} =0$. 

Using the splitting $X+A=\mu_{+}+\mu_{-}$, it is immediate to derive the commutator formulas (see \cite[Lemma 2.2]{Paternain2012} for instance)
\begin{align}
    \begin{split}
	[H, X + A] &= \pi_0 (X_\perp - A_V) + (X_\perp - A_V)\pi_0, \\
	[H,X_\perp -A_V] &= -(\pi_0 (X+A) + (X+A)\pi_0), \\
    \end{split}    
    \label{eq:commutatorA}
\end{align}
as well as the following identities, obtained by computing $[H^2,X+A]$ in two ways:
\begin{align}
  \pi_0 (X+A) = \pi_0 (X_\perp -A_V) H, \qquad (X+A) \pi_0 = - H(X_\perp -A_V) \pi_0.
  \label{eq:commutator2A}
\end{align}

%We also write those for the connection $\Aadj$, which read
%\begin{align*}
%    [H,X\Aadj] = \pi_0 (X_\perp + A_V^*) + (X_\perp + A_V^*) \pi_0, &\qquad [H,X_\perp + A_V^*] = -(\pi_0 (X\Aadj) + (X\Aadj)\pi_0), \\
%    \pi_0 (X\Aadj) = \pi_0 (X_\perp +A_V^*) H, &\qquad (X\Aadj) \pi_0 = - H(X_\perp +A_V^*) \pi_0.
%\end{align*}

\subsection{Decompositions of the data space} 

Denote $\rotpi (\x,v) = (\x,-v)$ the antipodal map (or rotation by $\pi$). A function $f$ is even/odd on $SM$ if $f\circ \rotpi  = +/- f$. Also define the {\em antipodal scattering relation} to be the mapping $\alpha_a:\partial (SM)\to \partial (SM)$
\begin{align*}
    \alpha_a = \alpha\circ \rotpi  = \rotpi  \circ \alpha \qquad (\alpha:\text{ scattering relation}).   
\end{align*}
$\alpha_a$ is an involution and $\alpha_a(\partial_\pm (SM))\subset \partial_\pm (SM)$. It is straightforward to see that the function $G(\x,v) := U_A(\x,-v)$ solves the transport problem
\begin{align*}
    (X+A) G = 0, \qquad G|_{\partial_+ SM} = C_A\circ \rotpi ,
\end{align*}
so that $G(\x,v) = U_A(\x,v) C_A (\rotpi (\varphi_-(\x,v)))$ for every $(\x,v)\in SM$. In particular, this implies the relation 
\begin{align*}
    U_A(\x,-v) = U_A(\x,v) C_A ( \varphi_+(\x,-v) ), \qquad (\x,v) \in SM.
\end{align*}
Writing this for $U_A(\x,v)$, we obtain the identity
\begin{align}
    C_A( \varphi_+(\x,v) )\  C_A( \varphi_+(\x,-v) ) = \Imm_n, \quad (\x,v)\in SM. 
    \label{eq:CAid}
\end{align}
We will use this to characterize the symmetries of the ray transforms over even and odd integrands.  In particular, the identity \eqref{eq:CAid} means that 
\begin{align*}
    C_A (\x,v)\ C_A( \alpha_a (\x,v)) = \Imm_n, \qquad (\x,v)\in \partial_- (SM).
\end{align*}
Note also the obvious identities
\begin{align}
    \begin{split}
	\varphi_t(\x,v) &= \rotpi ( \varphi_{\tau(\x,v)-t} (\alpha_a(\x,v)) ), \quad (\x,v)\in \partial_+ (SM), \quad t\in [0,\tau(\x,v)]. \\
	\tau(\x,v) &= \tau(\alpha_a(\x,v)), \qquad (\x,v)\in \partial_+ (SM).
    \end{split}    
    \label{eq:geoid}
\end{align}

\begin{lemma} \label{lem:sym}
    If $f\in C^\infty(SM)$ satisfies $f\circ \rotpi  = \pm f$, then the data $I_A f$ satisfies
    \begin{align*}
	I_A f(\alpha_a(\x,v)) = \pm C_A(\alpha(\x,v)) I_A f(\x,v).
    \end{align*}
\end{lemma}

\begin{proof} We only treat the case of $f$ even, the odd case being similar. We write
    \begin{align*}
	I_A f(\alpha_a(\x,v)) &= \int_0^{\tau(\alpha_a(\x,v))} U_A^{-1} (\varphi_t(\alpha_a(\x,v))) f(\varphi_t(\alpha_a(\x,v)))\ dt \\
	&\stackrel{\eqref{eq:geoid}}{=} \int_0^{\tau(\x,v)} U_A^{-1} (\rotpi ( \varphi_{\tau-t}(\x,v))) f(\rotpi ( \varphi_{\tau-t}(\x,v))) \ dt \\
	&\stackrel{u=\tau-t}{=} \int_0^{\tau(\x,v)} (U_A (\rotpi  \varphi_u (\x,v)))^{-1} f(\varphi_{u}(\x,v))\ du \\
	&= (C_A(\rotpi (\x,v)))^{-1} I_A f(\x,v),
    \end{align*}
    and the identity follows since $C_A(\rotpi (\x,v))^{-1} = C_A(\alpha(\x,v))$ by \eqref{eq:CAid}. The proof is complete.
\end{proof}
This motivates a decomposition of $C^\infty (\partial_+ (SM)) = \V_{A,+} \oplus \V_{A,-}$, where we define 
\begin{align*}
    \V_{A,\pm} := \{ h\in C^\infty (\partial_+ (SM)):\quad h(\alpha_a(\x,v)) = \pm C_A(\alpha(\x,v)) h(\x,v) \}.
\end{align*}
This decomposition is unique and given explicitely by $h = h_{A,+} + h_{A,-}$ with 
\begin{align*}
    h_{A,\pm}(\x,v) = \frac{1}{2} \left(h(\alpha_a(\x,v)) \pm C_A^{-1}(\alpha(\x,v)) h(\x,v) \right), \qquad (\x,v)\in \partial_+ (SM).
\end{align*}
Such symmetries, via extension as first integrals of $X+A$, generate even and odd functions on $SM$:

\begin{lemma}    \label{lem:sym2}
    If $h\in \V_{A,+}$ ($\V_{A,-}$) then the function $h_{\psi,A}$ is even (odd) on $SM$.
\end{lemma}

\begin{proof} Suppose $h\in \V_{A,+}$ (the case of $\V_{A,-}$ is similar). Then, for any $(\x,v)\in SM$,
    \begin{align*}
	h_{\psi,A}(\x,-v) &= U_A(\x,-v) h(\varphi_-(\x,-v)) \\
	&= U_A(\x,v) C_A(\rotpi  (\varphi_- (\x,v))) h(\alpha_a(\varphi_-(\x,v))) \\
	&= U_A(\x,v) C_A(\rotpi  (\varphi_- (\x,v))) C_A (\alpha (\varphi_-(\x,v))) h(\varphi_-(\x,v)) \\
	&= U_A(\x,v) \cdot \Imm_n \cdot h(\varphi_-(\x,v)) \\
	&= h_{\psi,A}(\x,v),
    \end{align*} 
    hence the proof.
\end{proof}

In general, the decomposition $\V_{A,+} \oplus \V_{A,-}$ is not orthogonal in $L^2_\mu (\partial_+ SM)$. In this context of non-Hermitian connections, the more natural relation is the following. 

\begin{lemma}\label{lem:orthogonal}
    For any $GL(n,\mathbb{C})$-connection $A$, the following decompositions hold, orthogonal in the $L^2_\mu (\partial_+(SM))$ sense:
    \begin{align*}
	C^\infty(\partial_+(SM)) = \V_{A,+} \stackrel{\perp}{\oplus} \V_{\Aadj,-} = \V_{\Aadj,+} \stackrel{\perp}{\oplus} \V_{A,-}.
    \end{align*}
    In particular, if the connection is unitary $A = \Aadj$, then $C^\infty(\partial_+(SM)) = \V_{A,+} \stackrel{\perp}{\oplus} \V_{A,-}$.
\end{lemma}

\begin{proof} It is enough to prove the first equality, as the second follows by considering the connection $\Aadj$. \\
    \noindent{\bf Uniqueness.} For $h\in C^\infty(\partial_+ (SM))$, a unique decomposition $h = h_+ + h_- \in \V_{A,+} \oplus \V_{\Aadj,-}$ is given by:
    \begin{align*}
	h_+(\x,v) &= \big(\Imm_n + C_A^* C_A(\alpha(\x,v))\big)^{-1} \Big(h(\x,v) + C_A^*(\alpha(\x,v))h(\alpha_a(\x,v))\Big), \\
	h_-(\x,v) &= \big(\Imm_n + C_A^{-1} (C_A^{-1})^* (\alpha(\x,v))\big)^{-1} \Big(h(\x,v) - C_A^{-1}(\alpha(\x,v))h(\alpha_a(\x,v))\Big),
    \end{align*}
    obtained by fulfilling the conditions of $\V_{A,+}$ and $\V_{\Aadj,-}$ and using that $C_{-A^*} = (C_A^*)^{-1}$ by virtue of \eqref{eq:UUstar}. 

    \noindent{\bf Orthogonality.} Let $h\in \V_{A,+}$ and $w\in \V_{\Aadj, -}$. We write
    \begin{align*}
	\int_{\partial_+ (SM)} \dprod{h}{w} (\x,v) \mu\ d\Sigma^2 &= \int_{\partial_+ (SM)} \frac{1}{\tau(\x,v)} \int_0^{\tau(\x,v)} \dprod{h}{w}(\varphi_- (\varphi_t(\x,v)))\ dt\ \mu\ d\Sigma^2.
    \end{align*}  
    We then write
    \begin{align*}
	\dprod{h(\varphi_- (\varphi_t(\x,v)))}{w(\varphi_-(\varphi_t(\x,v)))} &= \dprod{U_A^{-1} h_{\psi,A}}{U_{\Aadj}^{-1} w_{\psi, \Aadj}} (\varphi_t(\x,v)) \\
	&= \dprod{U_{\Aadj}^* h_{\psi,A}}{U_{\Aadj}^{-1} w_{\psi, \Aadj}} (\varphi_t(\x,v)) \\
	&= \dprod{h_{\psi,A}}{w_{\psi,\Aadj}}(\varphi_t(\x,v)).
    \end{align*}
    Plugging this last expression into the first equality and applying Santalo's formula, we arrive at 
    \begin{align*}
	\int_{\partial_+ (SM)} \dprod{h}{w} (\x,v) \mu\ d\Sigma^2 = \int_{SM} \frac{1}{\tau\circ \varphi_-} \dprod{h_{\psi,A}}{w_{\psi, \Aadj}}\ d\Sigma^3 = 0,
    \end{align*}
    since $\tau\circ \varphi_-$ is even in $v$ and by virtue of Lemma \ref{lem:sym2}, $h_{\psi,A}$ is even in $v$ and $w_{\psi,\Aadj}$ is odd in $v$. The lemma is proved.
\end{proof}

\subsection{Boundary operators} \label{sec:boundaryop}

Extending notation from \cite{Paternain2013a}, we define for $w\in C(\partial_+ (SM), \Cm^n)$
\begin{align}
    Q_{A,\pm} w(\x,v) := \left\lbrace
    \begin{array}{lr}
	w(\x,v) & \text{if } (\x,v) \in \partial_+ (SM), \\
	\pm C_A(\x,v) w\circ\alpha (\x,v) & \text{if } (\x,v) \in \partial_- (SM).
    \end{array}
    \right.
    \label{eq:defQ}
\end{align}
Note that $Q_+ w\in C(\partial(SM), \Cm^n)$. We also define
\begin{align}
    \begin{split}
	{\cal S}_A^\infty (\partial_+ (SM),\Cm^n) &:= \{ w\in C^\infty(\partial_+ (SM)): w_{A,\psi} \in C^\infty(SM,\Cm^n) \} \\
	&= \{ w\in C^\infty(\partial_+ (SM)): Q_{A,+}w \in C^\infty(\partial(SM),\Cm^n) \},	
    \end{split}    
    \label{eq:Sinf}
\end{align}
where the second equality is established in \cite[Lemma 5.1]{Paternain2013a}. This is due in part to the fact that 
\begin{align*}
    h_{\psi,A}|_{\partial SM} = Q_{A,+} h, \qquad h\in C^\infty(\partial_+ (SM)),
\end{align*}
so that the operator $Q_{A,+}:{\cal S}_A^\infty (\partial_+ (SM),\Cm^n)\to C^\infty(\partial_+ (SM),\Cm^n)$ makes sense. We also introduce $B_{A,\pm}:C^\infty(\partial (SM),\Cm^n)\to C^\infty(\partial_+ (SM),\Cm^n)$, 
\begin{align}
    B_{A,\pm} g (\x,v) = g(\x,v) \pm C_A^{-1} (\alpha(\x,v)) g(\alpha(\x,v)), \qquad (\x,v)\in \partial_+ (SM).
    \label{eq:defB}
\end{align}
(note the sign difference with \cite{Paternain2013a}). $B_{A,-}$ appears naturally in the fundamental theorem of calculus along a geodesic: for $(\x,v)\in \partial_+ (SM)$,
\begin{align*}
    I_A [(X+A) u] (\x,v) &= \int_0^{\tau(\x,v)} U_A^{-1} (\varphi_t(\x,v)) (X+A)u (\varphi_t(\x,v))\ dt \\
    &= \int_0^{\tau(\x,v)} X( U_A^{-1} u ) (\varphi_t(\x,v))\ dt \\
    &= \left[ U_A^{-1} u(\varphi_t(\x,v)) \right]_0^{\tau(\x,v)},
\end{align*}
so that 
\[ I_A [(X+A)u] = - B_{A,-} u|_{\partial (SM)}. \]

We state without proof the following straightforward claims.

\begin{lemma}\label{lem:claims}
    \begin{itemize}
	\item[$(i)$] If $h\in \V_{A,+}$, then $Q_{A,+} h$ is even in $v$ and $Q_{A,-} h$ is odd in $v$. 
	\item[$(ii)$] If $h\in \V_{A,-}$, then $Q_{A,+} h$ is odd in $v$ and $Q_{A,-} h$ is even in $v$. 
	\item[$(iii)$] If $q\in C^\infty(\partial (SM),\Cm^n)$ is even in $v$, then $B_{A,\pm} q \in \V_{A,\pm}$.
	\item[$(iv)$] If $q\in C^\infty(\partial (SM), \Cm^n)$ is odd in $v$, then $B_{A,\pm} q \in \V_{A,\mp}$.
    \end{itemize}    
\end{lemma}

%%%%%%%%%%%%%%%%%%%%%%%%%%%%%%%%%%%%%%%%% inversion formulas

\section{Inversion formulas and control of the error operators}\label{sec:inversion}

For $f:M\to \Cm^n$, $f\circ\pi:SM\to \Cm^n$ is an even function of $v$ and $(X_\perp - A_V)f$ is odd in $v$, hence from Lemma \ref{lem:sym} we have
\begin{align*}
  \text{Range } I_{A,0} \subset \V_{A,+}, \qquad \text{Range } I_{A,\perp} \subset \V_{A,-}.    
\end{align*}
We will see below that it is somehow natural to consider the inversion of operators $I_{A,0}$, $I_{A,\perp}$, $I_{\Aadj, 0}$ and $I_{\Aadj, \perp}$ together. 

\subsection{Adjoints}
For $f\in C^\infty(M,\Cm^n)$ and $h\in C^\infty(\partial_+ (SM), \Cm^n)$, we compute (denote $\dprod{\cdot}{\cdot}$ the Hermitian product on $\Cm^n$),
\begin{align*}
    \int_{\partial_+ SM} \dprod{h(\x,v)}{I_{A,0}f(\x,v)}\ \mu d\Sigma^2 &= \int_{\partial_+ SM} \int_0^{\tau(\x,v)} \dprod{h(\x,v)}{U_A^{-1}(\varphi_t(\x,v))f(\varphi_t(\x,v))}\ dt\ \mu d\Sigma^2 \\
    &\stackrel{\eqref{eq:UUstar}}{=} \int_{\partial_+ SM} \int_0^{\tau(\x,v)} \dprod{h(\x,v)}{U_{\Aadj}^*(\varphi_t(\x,v))f(\varphi_t(\x,v))}\ dt\ \mu d\Sigma^2 \\
  &= \int_{\partial_+ SM} \int_0^{\tau(\x,v)} \dprod{ U_{\Aadj}(\varphi_t(\x,v)) h(\x,v)}{f(\varphi_t(\x,v))}\ dt\ \mu d\Sigma^2 \\
  &= \int_{\partial_+ SM} \int_0^{\tau(\x,v)} \dprod{ h_{\psi,\Aadj} (\varphi_t(\x,v))}{f(\varphi_t(\x,v))}\ dt\ \mu d\Sigma^2 \\
  &= \int_{SM} \dprod{h_{\psi,\Aadj}}{f}\ d\Sigma^3 \qquad\qquad (\text{by Santal\'o's formula}) \\
  &= 2\pi \int_M \dprod{\pi_0 h_{\psi,\Aadj}}{f}\ dM. 
\end{align*}
So we deduce that 
\begin{align}
    I_{A,0}^* h = (2\pi) \pi_0 h_{\psi,\Aadj}.
    \label{eq:IA0star}
\end{align}
A similar argument with $f\in C^\infty_0(SM)$, and using the fact that $X_\perp^* = -X_\perp$ when either function in the $L^2(SM,\Cm^n)$ inner product vanishes at the boundary, yields that 
\begin{align}
  I_{A,\perp}^* h = -(2\pi) \pi_0 (X_\perp +A^{*}_V) h_{\psi,\Aadj}.
  \label{eq:IAperpstar}
\end{align}
%Applying these identities to $\Aadj$, we also derive that 
%\begin{align*}
%    I_{\Aadj, 0}^* h = (2\pi) \pi_0 h_{\psi, A}, \qquad I_{\Aadj,\perp}^* h = -(2\pi) \pi_0 (X_\perp + A_V^*) h_{\psi, A}. 
%\end{align*}

A direct use of Lemma \ref{lem:sym2} and inspection on symmetries yields the following 

\begin{lemma} \label{lem:adjoints}
  \[ I_{\Aadj, 0}^* (\V_{A,-}) = \{0\} \quad \text{ and } \quad I_{\Aadj, \perp}^* (\V_{A,+}) = \{0\}. \] 
\end{lemma}

\subsection{Fredholm equations for $I_{A,0}$ and $I_{A,\perp}$ - proof of Theorem \ref{thm:Fred}}

As discussed in the introduction, let us define the operators 
\begin{align}
    \begin{split}
	W_A f &:= \pi_0 (X_\perp - A_V) u_A^f, \qquad f\in C^\infty(M,\Cm^n), \\
	W_{A,\perp} f &:= \pi_0 u_A^{(X_\perp - A_V) f}, \qquad f\in C_0^\infty(M,\Cm^n).
    \end{split}    
    \label{eq:W}
\end{align}

We now prove Theorem \ref{thm:Fred} before studying the operators $W_{A}$ and $W_{A,\perp}$ further. For the proof below, let us make the comment that, in terms of solutions of elementary transport problems of the form $u_A^f$ and $h_{A,\psi}$ defined in Section \ref{sec:prelim}, the solution to the problem 
\begin{align*}
    (X+A)u = -f \quad (SM), \quad u|_{\partial_+ (SM)} = w,
\end{align*}
is $u = u_A^f + (w - I_{A}f)_{\psi,A}$, and the solution to the problem 
\begin{align*}
    (X+A)u = -f \quad (SM), \quad u|_{\partial_- (SM)} = h,
\end{align*}
is $u = u_A^f + ((C_A^{-1} h)\circ \alpha)_{\psi, A}$.

\begin{proof}[Proof of Theorem \ref{thm:Fred}] Here and below, for a function $u(\x,v)$ defined on $SM$, we write the even/odd decomposition with respect to $v$ as $u = u_+ + u_-$, where $u_{\pm}(\x,v):= (u(\x,v)\pm u(\x,-v))/2$. \\ 
    {\bf Inversion of $I_{A,0}$ (Proof of \eqref{eq:FredWA}).} Start from the equation 
    \begin{align*}
	(X+A) u_A^f = -f\quad (SM), \qquad u_A^f|_{\partial_- (SM)} = 0,    
    \end{align*}
    so that $u_A^f|_{\partial_+ (SM)} = I_{A,0} f$. Direct application of the left equation of \eqref{eq:commutator2A} to the transport equation gives
    \begin{align}
	f = \pi_0 f = -\pi_0 (X+ A) u_A^f = - \pi_0 (X_\perp -A_V) H u_A^f = - \pi_0 (X_\perp -A_V) H u_{A,-}^f,
	\label{eq:frcA}
    \end{align} 
    where the last step comes from the fact that $\pi_0 (X_\perp - A_V) H u_{A,+}^f = 0$. It now remains to write a transport problem for $Hu_{A,-}^f$, for which we use the formula for $[H,X+A]$:
    \begin{align*}
	(X+A)Hu_A^f &= \cancel{H(X+A) u_A^f} - [H,X+A]u_A^f \\   
	&= - \pi_0 (X_\perp -A_V)u_A^f - (X_\perp -A_V)\pi_0 u_A^f,  
    \end{align*}
    which upon projecting onto even harmonics, yields 
    \begin{align}
	(X+A) H u_{A,-}^f = - W_A f, \where\quad W_A f:= \pi_0 (X_\perp-A_V) u_A^f. 
	\label{eq:transWf}
    \end{align}
    This equation gives us $Hu_{A,-}^f = u_A^{W_Af} + h_{\psi,A}$, where 
    \begin{align*}
	h = (C_A^{-1} (H u_{A,-}^f|_{\partial_- (SM)}))\circ \alpha = \frac{1}{2} (B_{A,+} - B_{A,-}) (Hu_{A,-}^f|_{\partial (SM)}) = \frac{1}{4} (B_{A,+} - B_{A,-}) H Q_{A,-} I_{A,0}f.
    \end{align*}    
    Plugging this expression of $Hu_{A,-}^f$ back into \eqref{eq:frcA} yields the formula
    \begin{align*}
	f + W_A^2 f &= - \pi_0 (X_\perp-A_V) h_{\psi, A} \\
	&= \frac{1}{2\pi} I_{\Aadj,\perp}^* h \\
	&= \frac{1}{8\pi} I_{\Aadj,\perp}^* (B_{A,+} - B_{A,-}) H Q_{A,-} I_{A,0}f.    
    \end{align*}
    Finally, inspection of symmetries shows that $B_{A,-}HQ_{A,-} I_{A,0} f \in \V_{A,+}$, so that using Lemma \ref{lem:sym2} and the expression of $I_{\Aadj,\perp}^*$, it is annihilated by $I_{\Aadj,\perp}^*$. As a conclusion, we arrive at \eqref{eq:FredWA}. \\

    \noindent{\bf Inversion of $I_{A,\perp}$ (Proof of \eqref{eq:FredWAstar}).} Start from the equation 
    \begin{align*}
	(X+A) u_A^{(X_\perp -A_V)f} = -(X_\perp-A_V)f\quad (SM), \qquad u_A^{(X_\perp -A_V)f}|_{\partial_- (SM)} = 0,
    \end{align*}
    so that $u_A^{(X_\perp -A_V)f}|_{\partial_+ (SM)} = I_{A,\perp} f$. Note that since $I_{A,\perp}f \in \V_{A,-}$, $Q_{A,-} I_{A,\perp}f$ is even in $v$ on $SM$, and in particular we have
    \begin{align*}
	u_{A,+}^{(X_\perp - A_V)f}|_{\partial (SM)} = \frac{1}{2} Q_{A,-} I_{A,\perp}f.
    \end{align*}
    Applying the Hilbert transform to the transport equation above, using the commutators and projecting onto odd harmonics, we obtain
    \begin{align}
	(X+A) (Hu_{A,+}^{(X_\perp -A_V)f} - f) = - (X_\perp - A_V) W_{A,\perp} f, \qquad W_{A,\perp} f := \pi_0\ u_A^{(X_\perp-A_V)f}.
	\label{eq:transWAperp}
    \end{align}
    From this equation, we get that the function $(Hu_{A,+}^{(X_\perp -A_V)f} - f)$ is nothing but 
    \begin{align}
	Hu_{A,+}^{(X_\perp -A_V)f} - f = u_{A}^{(X_\perp - A_V) W_{A,\perp} f} + w_{\psi, A},
	\label{eq:HuAplus}
    \end{align}
    where 
    \begin{align*}
	w = (C_{A}^{-1} (Hu_{A,+}^{(X_\perp -A_V)f})|_{\partial_- SM})\circ \alpha = \frac{1}{4} (B_{A,+} - B_{A,-}) H Q_{A,-} I_{A,\perp} f.
    \end{align*}
    Upon applying $\pi_0$ (fiber average) to \eqref{eq:HuAplus}, we obtain
    \begin{align*}
	f + W_{A,\perp}^2 f = - \pi_0 w_{\psi, A} = -\frac{1}{2\pi} I_{\Aadj,0}^* w.
    \end{align*}
    As in the inversion of $I_{A,0}$, we notice that $B_{A,-} H Q_{A,-} I_{A,\perp} f\in \V_{A,-}$ and as such is annihilated by $I_{\Aadj,0}^*$. As a conclusion, the reconstruction formula, in its final form, looks like \eqref{eq:FredWAstar}. 
\end{proof}

\subsection{Properties of the error operators} 

Unlike the geodesic case studied in \cite{Pestov2004}, $W_A$ and $W_{A,\perp}$ are not always $L^2(M)$-adjoints. Consider the equations \eqref{eq:FredWA} and \eqref{eq:FredWAstar} corresponding to the connection $\Aadj$. They give,
\begin{align}
  f + W_{\Aadj}^2 f &= \frac{1}{8\pi} I_{A,\perp}^* B_{\Aadj,+} H Q_{\Aadj,-} I_{\Aadj,0}f, \qquad f\in C^\infty(M,\Cm^n). \label{eq:FredA2}\\
  f + W_{\Aadj,\perp}^2 f &= - \frac{1}{8\pi} I_{A, 0}^* B_{\Aadj,+} H Q_{\Aadj,-} I_{\Aadj,\perp}f, \qquad f\in C_0^\infty(M,\Cm^n).  \label{eq:FredAperp2}
\end{align}
Inspecting the right hand sides suggests that, for instance, taking the adjoint equation to \eqref{eq:FredWA} would yield \eqref{eq:FredAperp2}. A partial answer to this heuristic guess is to establish:
\begin{lemma} \label{lem:Wops}
    The operators $W_A$ and $W_{\Aadj,\perp}$ are $L^2(M,\Cm^n)\to L^2(M,\Cm^n)$ adjoints. As a consequence, so are $W_{\Aadj}$ and $W_{A,\perp}$. In particular, if $A = \Aadj$, then $W_A$ and $W_{A,\perp}$ are adjoints. 
\end{lemma}

\begin{proof} It is enough to check it for $f,g\in C^\infty_0(M)$. We compute
  \begin{align*}
    2\pi (W_A f, g)_{L^2(M)} &= 2\pi (\pi_0 (X_\perp -A_V) u_A^f, g)_{L^2(M)} \\
    &= ( (X_\perp -A_V) u_A^f, g)_{L^2(SM)} \\
    &= (u_A^f, -(X_\perp + A^*_V)g)_{L^2(SM)} \\
    &= (u_A^f, (X\Aadj)u_{\Aadj}^{(X_\perp + A^*_V)g})_{L^2(SM)} \\ 
    &= (-(X+A)u_A^f, u_{\Aadj}^{(X_\perp + A^*_V)g})_{L^2(SM)} - \cancel{(I_{A,0}f, I_{\Aadj,\perp}g)_{L^2_\mu (\partial_+ (SM))}} \\ 
    &= (f, u_{\Aadj}^{(X_\perp + A^*_V)g})_{L^2(SM)} \\
    &= 2\pi (f, \pi_0 u_{\Aadj}^{(X_\perp + A^*_V)g})_{L^2(M)} \\
    &= 2\pi (f, W_{\Aadj, \perp}g)_{L^2(SM)}.
  \end{align*}
  The crossed term is zero because $I_{A,0}f\in \V_{A,+}$ while $I_{\Aadj,\perp}g \in \V_{\Aadj, -}$ and both spaces are orthogonal by virtue of Lemma \ref{lem:orthogonal}. The lemma is proved. 
\end{proof}

The next result establishes that in the case where the metric is simple, the reconstruction formulas \eqref{eq:FredWA}, \eqref{eq:FredWAstar}, \eqref{eq:FredA2} and \eqref{eq:FredAperp2} are in fact Fredholm equations, as the operators $W_{A,0}$ and $W_{A,\perp}$ are compact. In order to prove this, we need to make explicit their Schwartz kernels, which in turns requires some recalls about Jacobi fields.

\paragraph{Jacobi fields and simplicity.} Variations of the exponential map are computed following \cite{Merry2011}. For $\xi\in T_{(\x,v)} SM$ uniquely written as $\xi = a(0) X_{(\x,v)} + b(0) X_{\perp(\x,v)} + c(0) V_{(\x,v)}$, there exist scalar functions $a(\x,v,t),b(\x,v,t),c(\x,v,t)$ such that (keeping $(\x,v)$ implicit)
\begin{align*}
  d\varphi_t(\xi) = a(t) X(t) + b(t) X_\perp (t) + c(t) V(t), 
\end{align*}
where $Y(t)$ denotes $Y(\varphi_t(\x,v))$ for $Y \in \{X,X_\perp,V\}$. Due to the structure equations, we deduce that $a,b,c$, defined on $\D$, solve the system
\begin{align*}
  \dot a = 0, \qquad \dot b + c = 0, \qquad \dot c - \kappa(\gamma(t)) b = 0.
\end{align*}
Particular Jacobi fields of interest are $d\varphi_t (X_\perp)$ described by $(a_1 \equiv 0, b_1, c_1)$ with initial condition $(a_1,b_1,c_1)(0) = (0,1,0)$, and $d\varphi_t (V)$ described by $(a_2\equiv 0, b_2, c_2)$ with initial condition $(a_2,b_2,c_2)(0) = (0,0,1)$. In particular, as stated in Section \ref{sec:main}, the absence of conjugate points on a surface $(M,g)$ is equivalent to the non-vanishing of the function $b_2$ outside $\{t=0\}$.

\paragraph{Kernels of $W_A$ and $W_{A,\perp}$.} We now make explicit the kernels of the operators $W_A$ and $W_{A,\perp}$ defined in \eqref{eq:W}, by showing the following
\begin{lemma} \label{lem:kernels} The operators $W_A, W_{A,\perp}$ take the form
    \begin{align*}
	W_A f(\x) &= \frac{1}{2\pi} \int_{S_\x} \int_0^{\tau(\x,v)} w_A(\x,v,t) f(\varphi_t(\x,v))\ dt\ dS(v), \qquad f\in C^\infty(M), \\
	W_{A,\perp} h(\x) &= \frac{1}{2\pi} \int_{S_\x} \int_0^{\tau(\x,v)} w_{A,\perp}(\x,v,t) h(\varphi_t(\x,v))\ dt\ dS(v), \qquad h\in C_0^\infty(M),
    \end{align*}
    with respective kernels, in exponential coordinates, given by 
    \begin{align}
	w_A(\x,v,t) &= \left(X_\perp - A_V - \frac{b_1}{b_2} V\right) E_A^{-1}(\x,v,t) - V\left( \frac{b_1}{b_2} \right) E_A^{-1}(\x,v,t),  \label{eq:wA} \\
	w_{A,\perp} (\x,v,t) &= E_A^{-1}(\x,v,t) \left( \frac{Vb_2}{b_2^2}(t) - A_V(\varphi_t(\x,v)) \right) - \frac{1}{b_2(t)} V(E_A^{-1}(\x,v,t)).         \label{eq:wAperp} 
    \end{align}    
\end{lemma}

The proof of Lemma \ref{lem:kernels} makes use of the following property, whose proof we relegate to the appendix:
\begin{lemma}\label{lem:tau} For every $(\x,v)\in SM$, 
    \[ b_2(\x,v,\tau(\x,v)) X_\perp \tau(\x,v) = b_1(\x,v,\tau(\x,v)) V\tau(\x,v). \]        
\end{lemma}

\begin{proof}[Proof of Lemma \ref{lem:kernels}] {\bf Proof of \eqref{eq:wA}.} Using the definition \eqref{eq:W},
    \begin{align}
	\begin{split}
	    W_A f(\x) &= \frac{1}{2\pi} \int_{S_\x} \int_0^{\tau(\x,v)} (X_\perp -A_V) E_A^{-1}(\x,v,t) f(\varphi_t(\x,v))\ dt\ dS(v) \\ 
	    &\qquad + \frac{1}{2\pi} \int_{S_\x} (X_\perp \tau) E_A^{-1}(\x,v,\tau) f(\varphi_\tau(\x,v))\ dS(v),    
	\end{split}
	\label{eq:tmpw}	
    \end{align}
    with $E_A$ defined in \eqref{eq:EA}. In this expression, the only term which differentiates $f$ is given by $X_\perp (f(\varphi_t(\x,v)))$, which we rewrite as 
    \begin{align*}
	X_\perp (f(\varphi_t(\x,v))) = b_1(\x,v,t) X_\perp f(\varphi_t(\x,v)) + c_1(\x,v,t) \cancel{ V f(\varphi_t(\x,v)) } = \frac{b_1(\x,v,t) }{ b_2(\x,v,t)} V (f(\varphi_t(\x,v))).
    \end{align*}
    The corresponding term can then be rewritten as
    \begin{align*}
	\int_{S_\x} \int_0^{\tau(\x,v)} &E_A^{-1}(\x,v,t) X_\perp (f(\varphi_t(\x,v)))\ dt\ dS(v) \\
	&= \int_{S_\x} \int_0^{\tau(\x,v)} E_A^{-1}(\x,v,t) \frac{b_1(\x,v,t) }{ b_2(\x,v,t)} V (f(\varphi_t(\x,v)))\ dt\ dS(v) \\
	&= \int_{S_\x} V\left( \int_0^{\tau(\x,v)} E_A^{-1}(\x,v,t) \frac{b_1(\x,v,t) }{ b_2(\x,v,t)} f(\varphi_t(\x,v))\ dt \right)\ dS(v) \\
	&\qquad - \int_{S_\x} (V\tau) E_A^{-1}(\x,v,\tau) \frac{b_1(\x,v,\tau) }{ b_2(\x,v,\tau)} f(\varphi_\tau)\ dS(v) \\
	&\qquad - \int_{S_\x} \int_0^{\tau(\x,v)} V \left( E_A^{-1}(\x,v,t) \frac{b_1(\x,v,t) }{ b_2(\x,v,t)} \right) f(\varphi_t(\x,v)) \ dt\ dS(v).
    \end{align*}
    In the last right-hand-side, the first term vanishes identically and the second cancels out the boundary term in \eqref{eq:tmpw} thanks to Lemma \ref{lem:tau}. We then arrive at an expression for $w_A$ as 
    \begin{align*}
	w_A(\x,v,t) = (X_\perp - A_V) E_A^{-1}(\x,v,t) - V\left( \frac{b_1(\x,v,t)}{b_2(\x,v,t)} E_A^{-1} (\x,v,t)\right),
    \end{align*}
    which yields \eqref{eq:wA} after applying a product rule. 

    {\bf Proof of \eqref{eq:wAperp}.} On to $W_{A,\perp}$, we write
    \begin{align*}
	W_{A,\perp} h (\x) = \frac{1}{2\pi} \int_{S_\x} \int_0^{\tau(\x,v)} E_A^{-1}(\x,v,t)(X_\perp h (\varphi_t(\x,v)) - A_V (\varphi_t(\x,v)) h(\varphi_t(\x,v))) \ dt\ dS(v).
    \end{align*}
    We rewrite the only term which differentiates $f$ as
    \begin{align*}
	X_\perp h(\varphi_t(\x,v)) = \frac{1}{b_2(t)} (b_2(t) X_\perp h (\varphi_t(\x,v)) + c_2(t) \cancel{Vh(\varphi_t(\x,v)}) = \frac{1}{b_2(t)} V (h\circ \varphi_t(\x,v)).  
    \end{align*}
    Integrating this term by parts on $S_\x$ in the expression of $W_{A,\perp} h(\x)$, this creates a boundary term of the form
    \begin{align*}
	\frac{-1}{2\pi} \int_{S_\x} \frac{E_A^{-1}(\x,v,\tau(\x,v))}{b_2(\x,v,\tau(\x,v))} (V\tau) h(\varphi_\tau(\x,v))\ dS(v),
    \end{align*}
    which vanishes since by assumption $h\in C_0^\infty(M)$. For the remaining term, we obtain the expression for $w_{A,\perp}$ as
    \begin{align*}
	w_{A,\perp}(\x,v,t) &= -V\left( \frac{E_A^{-1}(\x,v,t)}{b_2(t)} \right) - E_A^{-1}(\x,v,t) A_V(\varphi_t(\x,v)),
    \end{align*}    
    hence \eqref{eq:wAperp}. 
\end{proof}

\paragraph{The operators $W_A$ and $W_{A,\perp}$ are compact.}

%In what follows, the operators we will be studying are of the form $K:L^2(M,\Cm^n)\to L^2(M,\Cm^n)$ with expression
%\begin{align*}
%    Kf(\x) = \frac{1}{2\pi} \int_{S_\x} \int_0^{\tau(\x,v)} k(\x,v,t) f(\varphi_t(\x,v))\ dt\ dS(v),
%\end{align*}
%and we call $k$ the ``kernel of $K$ up to exponential map'', since for a simple surface, the mapping $(v,t)\mapsto \y = \exp_\x (v,t)$ is a diffeomorphism onto $M$ with inverse map $\y\mapsto (v_\x(\y), t_\x(\y) = d_g(\x,\y))$ and change of volume $b_2\ dt\ dS(v) = \exp_\x^* (dM_\y)$, so that
%\begin{align*}
%    Kf(\x) = \frac{1}{2\pi} \int_M \K(\x,\y) f(\y)\ dM_\y, \qquad \K(\x,\y) := \frac{k(\x,v_\x(\y),t_\x(\y))}{b_2(\x,v_\x(\y),t_\x(\y))}. 
%\end{align*}

%It is then immediate to see that a sufficient condition to make $K$ compact is that $\|\K\|\in L^2(M\times M)$ where $\|\|$ denotes Frobenius norm. In terms of the function $k$, this reads
%\begin{align}
%    \int_{M} \int_{S_\x} \int_0^{\tau(\x,v)} \frac{\|k(\x,v,t)\|^2}{b_2(\x,v,t)}\ dt\ dS(v)\ dM_\x <\infty \qquad \implies \qquad K \text{ compact}. 
%    \label{eq:criterion}
%\end{align}

In what follows, for a $\Cm^{n\times n}$ matrix $B$, we denote $\|B\| = (\tr (B^* B))^{\frac{1}{2}}$ its Frobenius norm. For an operator of the form
\begin{align}
    Wf(\x) = \frac{1}{2\pi} \int_{S_\x} \int_0^{\tau(\x,v)} \frac{w(\x,v,t)}{b_2(\x,v,t)} f(\varphi_t(\x,v))\ b_2(\x,v,t)\ dt\ dS(v) = \int_M {\cal W}(\x,y) f(\y)\ dM_\y,  
    \label{eq:op}
\end{align}
where we have defined ${\cal W} (\x,\y) := \frac{1}{2\pi} \frac{w(\x,Exp_\x^{-1} (\y))}{b_2(\x, Exp_\x^{-1} (\y))}$, we may obtain an estimate on the $L^2(M,\Cm^n)\to L^2(M,\Cm^n)$ norm of $W$ by computing 
\begin{align*}
    \|W\|_{L^2\to L^2}^2 &= \int_M \int_M {\|{\cal W}\|_\rho^2}(\x,\y)\ dM_\x\ dM_\y, 
\end{align*}
whenever the right hand side is finite, and where $\|\cdot\|_\rho$ denotes the spectral norm on $\Cm^{n\times n}$. Using that $\|\cdot\|_\rho\le \|\cdot\|$ and changing variable $\y = Exp_\x(v,t)$, we arrive at the following estimate, to be used below
\begin{align}
    \|W\|_{L^2\to L^2}^2\le \frac{1}{4\pi^2} \int_M \int_{S_\x} \int_0^{\tau(\x,v)} \frac{\|w\|^2(\x,v,t)}{b_2(\x,v,t)}\ dt\ dS(v)\ dM_\x,
    \label{eq:basicest}
\end{align}
which implies both continuity and compactness of $W$ whenever the right hand side is finite.  

\begin{lemma} \label{lem:compact} The operators $W_A$ and $W_{A,\perp}$ (and by duality via Lemma \ref{lem:Wops}, $W_{\Aadj}$ and $W_{\Aadj, \perp}$) are $L^2(M,\Cm^n)\to L^2(M,\Cm^n)$ compact. 
\end{lemma}

\begin{proof} The proof mainly consists in looking at the behavior of $w_A$ and $w_{A,\perp}$ defined in \eqref{eq:wA} and \eqref{eq:wAperp} near $t =0$. Near $t=0$, the following expansions hold: 
    \begin{align*}
	b_2(t) = -t + \O(t^3), \qquad E_A^{-1}(\x,v,t) = \Imm_n + t A(\x,v) + \O(t^2), 	
    \end{align*}
    where we used that 
    \begin{align*}
	\left. \frac{d}{dt}\right|_{t=0} \!\!\! E_A^{-1} (\x,v,t) = E_A^{-1} (\x,v,0) A(\varphi_0(\x,v)) = A(\x,v). 
    \end{align*}
    In particular, $V(E_A^{-1}(\x,v,t))/b_2(\x,v,t) = -A_V(\x,v) + \O (t)$. Together with the fact that the functions $\frac{Vb_2}{b_2^2}$ and $V\left( \frac{b_1}{b_2} \right)$ vanish as $t\to 0$ (see for instance \cite{Monard2013a,Pestov2004}), this allows to deduce that 
    \begin{align*}
	\lim_{t\to 0} w_{A}(\x,v,t) = \lim_{t\to 0} w_{A,\perp}(\x,v,t) = 0. 
    \end{align*}
    This means in particular that the function $\frac{k(\x,v,t)}{b_2(\x,v,t)}$ where $k \in \{w_A, w_{A,\perp}\}$ is bounded near $t=0$. Since it inherits the regularity of $b_1, b_2, A$ outside $t=0$ and $b_2$ does not vanish outside $\{t=0\}$ because $(M,g)$ is simple, the only problem was at $t=0$. For each operator, changing variables $\y(v,t) = \gamma_{\x,v}(t)$ with change of volume $dM_\y = |b_2(\x,v,t)|\ dt\ dS(v)$, the Schwarz kernels of $W_A$ and $W_{A,\perp}$ are of the form $\frac{k(\x,v(\y),d(\x,\y))}{b_2(\x,v(\y),d(\x,\y))}$ with $k\in \{w_A,w_{A,\perp}\}$, and they are bounded near the diagonal and away from the diagonal. Since $M\times M$ has finite volume, these kernels belong to $L^2(M\times M)$, and thus the operators $W_A, W_{A,\perp}: L^2(M,\Cm^n)\to L^2(M,\Cm^n)$ are compact. 
\end{proof}

\subsection{Analytic Fredholm approach - Proof of Theorem \ref{thm:analytic}}

We start with some preliminary estimates for ordinary differential equations. Let us define, for $A\in C^1(M,(\Lambda^1)^{n\times n})$,
\begin{align*}
    \alpha_A := \sup_{(\x,v)\in SM} \|(A+A^*)/2\|(\x,v).
\end{align*}
In addition, for a function $B:\D\to \Cm^{n\times n}$, we define $\|B\|_{F,\infty} := \sup_{(\x,v,t)\in \D} \|B(\x,v,t)\|$, making $(C^0(\D,\Cm^{n\times n}),\|\cdot\|_{F,\infty})$ into a Banach space. Moreover, since the Frobenius norm is submultiplicative, so is $\|\cdot\|_{F,\infty}$. We also consider the Banach space $(C^1(\D,\Cm^{n\times n}), \|\cdot\|_{F,\infty,1})$ with the norm
\begin{align*}
    \|B\|_{F,\infty,1} := \|B\|_{F,\infty} + \|X_\perp B\|_{F,\infty} + \|VB\|_{F,\infty} + \|XB\|_{F,\infty} + \|dB/dt\|_{F,\infty},     
\end{align*}
also submultiplicative. Now for a problem of the form 
\begin{align*}
    \frac{d}{dt} U + A(\varphi_t) U &= F \quad (\D), \qquad U|_{t=0} = 0,
\end{align*}
{\it a priori} estimates yield an estimate of the form 
\begin{align}
    \|U\|_{F,\infty} \le C(\alpha_A, \tau_\infty) \|F\|_{F,\infty}, \qquad C(\alpha_A,\tau_\infty) = \left\{
    \begin{array}{cc}
	\frac{\exp(\alpha_A \tau_\infty)-1}{\alpha_A} & \text{if } \alpha_A >0, \\
	\tau_\infty & \text{if } \alpha_A = 0.
    \end{array}
    \right.
    \label{eq:est1}
\end{align}
Moreover, for $Y \in \{X,X_\perp,V\}$, one may derive ODE's for $YU$ of the form 
\begin{align*}
    \frac{d}{dt} (Y U) + A(\varphi_t) (Y U) &= Y F - Y( A(\varphi_t)) U, \qquad YU|_{t=0} = 0, 
%    \frac{d}{dt} (V U) + A(\varphi_t) (V U) &= V F - b_2 X_\perp A(\varphi_t) U - c_2 A_V(\varphi_t) U, \qquad VU|_{t=0} = 0,
\end{align*}
which upon using \eqref{eq:est1} implies an estimate of the form 
\begin{align}
    \|U\|_{F,\infty,1} \le C(\alpha_A,\tau_\infty) (1 + C' \|A\|_{C^1(M,(\Lambda^1)^{n\times n})}) \|F\|_{F,\infty,1},
    \label{eq:est2}
\end{align}
where $C'$ is independent of $U$, $A$ or $F$. Similarly, a problem of the form 
\begin{align*}
    \frac{d}{dt} U + A(\varphi_t) U &= 0 \quad (\D), \qquad U|_{t=0} = \Imm_n,
\end{align*}
is equivalent to a problem for $W = U - \Imm_n$:
\begin{align*}
    \frac{d}{dt} W + A(\varphi_t) W &= -A (\varphi_t) \quad (\D), \qquad W|_{t=0} = 0,
\end{align*}
for which \eqref{eq:est2} applies. Combining this with the triangle inequality, and using that $\|A(\varphi_t)\|_{F, \infty, 1}\le \|A\|_{C^1(M,(\Lambda^1)^{n\times n})}$, we arrive at: 
\begin{align}
    \|U\|_{F,\infty,1} \le \sqrt{n} + C(\alpha_A,\tau_\infty) (1 + C' \|A\|_{C^1(M,(\Lambda^1)^{n\times n})}) \|A\|_{C^1(M,(\Lambda^1)^{n\times n})}.
    \label{eq:est3}
\end{align}
Finally, let us note that estimates \eqref{eq:est1}, \eqref{eq:est2} and \eqref{eq:est3} also hold if the connection is right-multiplied in the ODEs considered instead of left-multiplied. With these estimates in mind, we are ready to prove Theorem \ref{thm:analytic}. 

\begin{proof}[Proof of Theorem \ref{thm:analytic}] We prove the statement for $W_{A_\lambda}$ only, as the proof from $W_{A_\lambda,\perp}$ is similar. Recall that $w_A$, the kernel of $W_A$ up to exponential map, is given by 
    \begin{align*}
	w_A (\x,v,t) = \left(X_\perp - A_V - \frac{b_1}{b_2}V\right) E_A^{-1} - V\left( \frac{b_1}{b_2} \right) E_A^{-1},  
    \end{align*}
    with $E_A$ as defined in \eqref{eq:EA}. Estimates on $W_{A_\lambda}$ boil down to studying how $\lambda\mapsto w_{A_\lambda}$ behaves in the $C^0(\D,\Cm^{n\times n})$ topology, which in turn requires to look at how $\lambda\mapsto E_{A_\lambda}$ behaves in the $C^1(\D,\Cm^{n\times n})$ topology. Denote $E_\lambda = E_{A_\lambda}$ for short. Fix $\lambda_0\in \Cm$ and consider $\lambda$ close to $\lambda_0$, write by assumption 
    \begin{align*}
	A_\lambda = A_{\lambda_0} + (\lambda-\lambda_0) A'_{\lambda_0} + (\lambda-\lambda_0) B_\lambda, \qquad \lim_{\lambda\to \lambda_0} \|B_\lambda\|_{C^1(M, (\Lambda^1)^{n\times n})} = 0,
    \end{align*}
    and where $A'_{\lambda_0} \in C^1(M, (\Lambda^1)^{n\times n})$. Define $E'_{\lambda_0}$ the unique solution to 
    \begin{align*}
	\frac{d}{dt} E'_{\lambda_0} (t) + A_{\lambda_0} (\varphi_t) E'_{\lambda_0}(t) = - A'_{\lambda_0}(\varphi_t) E_{\lambda_0}(t) \qquad (\D), \qquad E'_{\lambda_0}|_{t=0} = 0.
    \end{align*}
    since $A'_{\lambda_0}(\varphi_t) E_{\lambda_0}(t)\in C^1(\D,\Cm^{n\times n})$, \eqref{eq:est2} gives us that $E'_{\lambda_0}\in C^1(\D,\Cm^{n\times n})$. Moreover, we have the relation $E_\lambda = E_{\lambda_0} + (\lambda-\lambda_0) E'_{\lambda_0} + (\lambda-\lambda_0) F_\lambda$, where $F_\lambda$ satisfies: 
    \begin{align*}
	\frac{d}{dt} F_\lambda + A_{\lambda_0}(\varphi_t) F_\lambda = - B_\lambda(\varphi_t) E_\lambda - A'_{\lambda_0}(\varphi_t) (E_\lambda-E_{\lambda_0}) \qquad (\D), \qquad F_\lambda|_{t=0} = 0.
    \end{align*}
    Since $\lim_{\lambda\to \lambda_0} \|B_\lambda(\varphi_t) E_\lambda + A'_{\lambda_0}(\varphi_t) (E_\lambda-E_{\lambda_0})\|_{C^1(\D,\Cm^{n\times n})} = 0$, then estimate \eqref{eq:est2} implies that $\lim_{\lambda\to \lambda_0} \|F_\lambda\|_{C^1(\D,\Cm^{n\times n})} = 0$ thus $\lambda\mapsto E_\lambda$ is an analytic $C^1(\D,\Cm^{n\times n})$ function. Similarly, we obtain
    \begin{align*}
	E_\lambda^{-1} = E^{-1}_{\lambda_0} - (\lambda-\lambda_0) E_{\lambda_0}^{-1} E'_{\lambda_0} E_{\lambda_0}^{-1} + (\lambda-\lambda_0) G_\lambda, \qquad \lim_{\lambda\to \lambda_0} \|G_\lambda\|_{F,\infty,1} = 0.
    \end{align*}
    Thus, upon defining 
    \begin{align*}
	w'_{A_{\lambda_0}} &:= \left(- X_\perp + (A_{\lambda_0})_V + V\left( \frac{b_1}{b_2} \right) + \frac{b_1}{b_2} V\right) E_{\lambda_0}^{-1} E'_{\lambda_0}E_{\lambda_0}^{-1} - (A'_{\lambda_0})_V E_{\lambda_0}, \\
	v_\lambda &:= \left(X_\perp - (A_{\lambda})_V - V\left( \frac{b_1}{b_2} \right) - \frac{b_1}{b_2} V\right) G_\lambda - (B_\lambda)_V (E_{\lambda_0}^{-1} - (\lambda-\lambda_0) E_{\lambda_0}^{-1} E'_{\lambda_0} E_{\lambda_0}^{-1}) \dots\\
	&\qquad + (A'_{\lambda_0})_V (\lambda-\lambda_0) E_{\lambda_0}^{-1} E'_{\lambda_0} E_{\lambda_0}^{-1},
    \end{align*}    
    we obtain that $w_{A_{\lambda_0}}'\in C^0(\D,\Cm^{n\times n})$ and 
    \begin{align}
	\frac{w_{A_\lambda} - w_{A_{\lambda_0}}}{\lambda-\lambda_0} &= w'_{A_{\lambda_0}} + v_\lambda, \qquad \lim_{\lambda\to \lambda_0} \|v_\lambda\|_{F,\infty} = 0,
	\label{eq:wexp}
    \end{align}
    where the estimate on $v_\lambda$ easily follows from the estimates on $G_\lambda$ and $B_\lambda$. In addition, let us analyze the behavior of $w'_{A_{\lambda_0}}$ near $t=0$ since the ratio $w'_{A_{\lambda_0}}/b_2$ will be continuous, hence bounded, elsewhere. Looking at the ODE satisfied by $E'_{\lambda_0}$, we have, near $t=0$, the expansions 
    \begin{align*}
	E_{\lambda_0} (\x,v,t) &= \Imm_n - t A_{\lambda_0}(\x,v) + \O(t^2),  & b_1(\x,v,t) = 1 + \O(t^2) \\
	E'_{\lambda_0} (\x,v,t) &= - t A'_{\lambda_0}(\x,v) + \O(t^2),  & b_2(\x,v,t) = -t + \O(t^3) \\
	V E'_{\lambda_0} (\x,v,t) &= - t (A'_{\lambda_0})_V + \O(t^2).
    \end{align*}
    With the additional vanishing of $V(b_1/b_2)$ as $t\to 0$, this is enough to establish that $\lim_{t\to 0} w'_{A_{\lambda_0}} = 0$, and from the relation \eqref{eq:wexp}, we also have $\lim_{t\to 0} v_\lambda = 0$. Thus $w'_{A_{\lambda_0}}/b_2$ and $v_\lambda/b_2$ are bounded on $\D$, and since $\D$ has finite volume, this clearly implies
    \begin{align*}
	\int_{M}\int_{S_\x} \int_0^{\tau(\x,v)} \frac{\|w'_{A_{\lambda_0}}(\x,v,t)\|^2}{b_2(\x,v,t)}\ dt\ dS(v)\ dM_\x <\infty,
    \end{align*}
    which, by virtue of estimate \eqref{eq:basicest}, implies that the operator $W'_{A_{\lambda_0}}$ defined in terms of the kernel $w'_{A_{\lambda_0}}$ as in \eqref{eq:op}, is $L^2(M,\Cm^n)\to L^2(M,\Cm^n)$ continuous. Reasoning similary on $V_\lambda$, we also obtain that 
    \begin{align*}
	\lim_{\lambda\to \lambda_0} \left\|\frac{W_{A_\lambda}-W_{A_{\lambda_0}}}{\lambda-\lambda_0} - W'_{A_{\lambda_0}}\right\|_{L^2\to L^2} = 0.	
    \end{align*}
    Theorem \ref{thm:analytic} is proved.
\end{proof}

\subsection{Further error estimates - proof of Theorem \ref{thm:estimate}} \label{sec:estimates}

We now refine the previous result by estimating the operator norms of $W_A$ and $W_{A,\perp}$ explicitly. 
%For conciseness, since $A$ is fixed, let us denote $E\equiv E_A$, defined as in \eqref{eq:EA}. Using $E(\x,v,t)$ as integrating factor, the solution $K(\x,v,t)$ of an ODE of the form
%\[  \frac{d}{dt} K(\x,v,t) - K(\x,v,t) A(\varphi_t(\x,v)) = S(\x,v,t), \qquad K(\x,v,0) = 0,      \]
%can be written as
%\begin{align}
%    K(\x,v,t) = \int_0^t S(\x,v,s) E (\x,v,s)\ ds\  E^{-1}(\x,v,t).    
%    \label{eq:ODEsol}
%\end{align}
In particular, we now define two functions of interest which appeared in the proof of Lemma \ref{lem:compact}: 
\begin{align}
    \begin{split}
	K_1(\x,v,t) &:= (X_{\perp (\x,v)} - A_V (\x,v)) E^{-1}(\x,v,t) + b_1(\x,v,t) E^{-1}(\x,v,t) A_V (\varphi_t(\x,v)) \\
	K_2(\x,v,t) &:= V_{(\x,v)} E^{-1}(\x,v,t) + b_2(\x,v,t) E^{-1}(\x,v,t) A_V (\varphi_t(\x,v)),	
    \end{split}
    \label{eq:K12}    
\end{align}
in terms of which the kernels $w_A$ and $w_{A,\perp}$ are written as
\begin{align}
    \begin{split}
	w_A(\x,v,t) &= K_1(\x,v,t) - \frac{b_1(\x,v,t)}{b_2(\x,v,t)} K_2(\x,v,t) - V\left( \frac{b_1}{b_2} \right) E^{-1}(\x,v,t), \\
	w_{A,\perp}(\x,v,t) &= \frac{-1}{b_2(\x,v,t)} K_2(\x,v,t) - V\left( \frac{1}{b_2} \right) E^{-1}(\x,v,t).	
    \end{split}
     \label{eq:wk}
\end{align}
Using the fact that
\[  \star F_A(\x) = X A_V + X_\perp A + [A,A_V], \]
we now establish the following
\begin{lemma} \label{lem:ODEK}
    The functions $K_\ell$ for $\ell= 1,2$ satisfy the following ODEs on $\D$: 
    \begin{align}
	\begin{split}
	    \frac{d}{dt} K_\ell(\x,v,t) - K_\ell(\x,v,t) A(\varphi_t(\x,v)) &= b_\ell(\x,v,t) E^{-1}(\x,v,t) \star F_A(\varphi_t(\x,v)), \\
	    K_\ell(\x,v,0) &= 0.
	\end{split}	
	\label{eq:ODEK}
    \end{align}        
\end{lemma}

\begin{proof}
    We only treat $K_2$, as the case of $K_1$ is similar. We compute directly, keeping the variables $(\x,v)$ implicit and writing $\dot{E} \equiv \frac{dE}{dt}$:
    \begin{align*}
	\frac{d}{dt} K_2(\x,v,t) &= V \dot{E}^{-1}  - c_2 E^{-1} A_V(\varphi_t) + b_2 \dot{E}^{-1} A_V (\varphi_t) + b_2 E^{-1} X A_V (\varphi_t) \\
	&= V (E^{-1} A(\varphi_t)) - c_2 E^{-1} A_V(\varphi_t) + b_2 E^{-1} A(\varphi_t) A_V(\varphi_t) + b_2 E^{-1} X A_V (\varphi_t) \\
	&= V(E^{-1}) A(\varphi_t) + E^{-1} (b_2 X_\perp A (\varphi_t) + c_2 A_V (\varphi_t)) - c_2 E^{-1} A_V(\varphi_t) \dots \\
	&\qquad\qquad + b_2 E^{-1} A(\varphi_t) A_V(\varphi_t) + b_2 E^{-1} X A_V (\varphi_t) \\
	&= (K_2 -b_2 E^{-1} A_V(\varphi_t)) A(\varphi_t) + E^{-1} b_2 X_\perp A (\varphi_t) \dots\\
	&\qquad \qquad + b_2 E^{-1} A(\varphi_t) A_V(\varphi_t) + b_2 E^{-1} X A_V (\varphi_t) \\
	&= K_2 A(\varphi_t) + b_2 E^{-1} (X_\perp A(\varphi_t) + X A_V (\varphi_t) + A(\varphi_t) A_V(\varphi_t) - A_V(\varphi_t)A(\varphi_t)) \\
	&= K_2 A(\varphi_t) + b_2 E^{-1} \star F_A(\varphi_t).\end{align*}
    Note that $b_2$ is a scalar function, so we can commute it. We also have used that $\dot b_2 = -c_2$. The proof is complete.
\end{proof}

The next result, whose proof we relegate to the Appendix, is an explicit bound on the quantities $V\left( \frac{b_1}{b_2} \right)$ and $V\left( \frac{1}{b_2} \right)$. Such quantities first appeared in \cite{Pestov2004} and arise as the kernels, up to exponential map, of the error operators $W$ and $W^*$ of the geodesic ray transform without connection. In what follows, we recall that for $(M,g)$ a simple surface, we may define $C_1(M,g) := \min_{\D} \frac{|b_2(\x,v,t)|}{t} >0$ and $C_2(M,g) := \max_{\D} \frac{|b_2(\x,v,t)|}{t}>0$.

\begin{lemma} \label{lem:Wkernelbound}
    Let $(M,g)$ a simple surface with constants $C_1, C_2$ as in \eqref{eq:simpleclaim}. Then the functions $V\left( \frac{b_1}{b_2} \right)$ and $V\left( \frac{1}{b_2} \right)$ satisfy the following estimates: 
    \begin{align*}
	\left| V\left( \frac{b_1}{b_2} \right)(\x,v,t) \right|, \left| V\left( \frac{1}{b_2} \right)(\x,v,t) \right| \le \frac{\|d\kappa\|_\infty C_2^3 t^2}{12 C_1^2}, \qquad (\x,v)\in SM, \quad t\in [0,\tau(\x,v)].
    \end{align*}
\end{lemma}

We now prove the main result of this section, Theorem \ref{thm:estimate}.

%Note additionally that since $b_2(\x,v,0) = 0$, $K_2$ also satisfies that $\frac{dK_2}{dt}(\x,v,0) = 0$ so that $\lim_{t\to 0} \frac{K_2(\x,v,t)}{b_2(\x,v,t)} = 0$. 

\begin{proof}[Proof of Theorem \ref{thm:estimate}] In order to apply estimate \eqref{eq:basicest} to $W_A$ and $W_{A,\perp}$, we now bound the functions $w_A$ and $w_{A,\perp}$ using expressions in \eqref{eq:wk}. The equation satisfied by $E = E_A$ implies 
    \begin{align*}
	\frac{d}{dt} \|E(\x,v,t)\|  \le \left\|(A + A^*)/2\right\| (\varphi_t(\x,v)) \|E(\x,v,t)\|, 
    \end{align*} 
    so that we may obtain the estimate
    \begin{align*}
	\|E(\x,v,t)\| \le e^{\alpha_A t} \|E(\x,v,0)\| = \sqrt{n} e^{\alpha_A t}.
    \end{align*}
    The same estimate holds for $E^{-1}$. Integrating \eqref{eq:ODEK} using $E$ as integrating factor, we deduce the following integral representations (we keep $(\x,v)$ implicit)
    \begin{align*}
	K_\ell(t) &= \int_0^t b_\ell(s) E^{-1}(s) \star F_A (\varphi_s) E(s)\ ds\ E^{-1}(t), \qquad \ell = 1,2, \\
	(b_2 K_1 - b_1 K_2)(t) &= \int_0^t (b_2(t)b_1(s)-b_1(t)b_2(s)) E^{-1}(s) \star F_A (\varphi_s) E(s)\ ds\ E^{-1}(t) \\
	&= \int_0^t b_2(\varphi_s,t-s) E^{-1}(s) \star F_A (\varphi_s) E(s)\ ds\ E^{-1}(t),
    \end{align*}
    where in the last equality, we have used \eqref{eq:cocycle} (proved in the appendix). We now bound the Frobenius norm of the left hand sides, using submultiplicativity of $\|\cdot\|_F$:
    \begin{align*}
	\|K_2\| (t) &\le \int_0^t |b_2|(s) \|E^{-1}(s)\| \|\star F_A(\varphi_s)\| \|E(s)\|\ ds \|E^{-1}(t)\| \\
	&\le \|\star F_A\|_\infty \int_0^t (C_2 s) (\sqrt{n} e^{\alpha_A s})^2\ ds \sqrt{n} e^{\alpha_A t} \le  n^{3/2} e^{3\alpha_A \tau_\infty} C_2 \|\star F_A\|_\infty \frac{t^2}{2}. 
    \end{align*}
    Similarly, using that $|b_2(\varphi_s,t-s)|\le C_2(t-s)$, we can arrive at the exact same bound for $\|b_2K_1-b_1K_2\| (t)$. Given the form of $w_A$ and $w_{A,\perp}$ in \eqref{eq:wk}, and the fact that bounds on $V\left( \frac{b_1}{b_2} \right)$ and $V\left( \frac{1}{b_2} \right)$ are the same and bounds on $\|K_2\|$ and $\|b_2 K_1 - b_1 K_2\|$ are the same, this will yield the same bound on $w_A$ or $w_{A,\perp}$. Therefore, let us focus on $w_{A,\perp}$: using the previous bound together with Lemma \ref{lem:Wkernelbound}
    \begin{align*}
	\|w_{A,\perp}(\x,v,t)\| &\le \frac{1}{|b_2|(t)} \|K_2(t)\| + \left| V\left( \frac{1}{b_2} \right)\right| \|E^{-1}(t)\| \\
	&\le \frac{n^{3/2}}{2} e^{3\alpha_A \tau_\infty} \frac{C_2}{C_1} \|\star F_A\|_\infty t + \|d\kappa\|_\infty \frac{C_2^3}{12 C_1^2} t^2 \sqrt{n} e^{\alpha_A \tau_\infty}. 
    \end{align*}
    Using $(a+b)^2\le 2(a^2 + b^2)$ to bound $\|w_{A,\perp}(\x,v,t)\|^2$ and using \eqref{eq:basicest}, we arrive at
    \begin{align*}
	\|W_{A,\perp}\|^2_{L^2\to L^2} &\le \frac{1}{4\pi^2} \int_M \int_{S_\x} \int_0^{\tau(\x,v)} \frac{\|w_{A,\perp}(\x,v,t)\|^2}{|b_2(\x,v,t)|}\ dt\ dS(v)\ dM_\x \\
	&\le \frac{\Vol M}{2\pi} \int_0^{\tau_\infty} \left(\|\star F_A\|_\infty^2 n^3 e^{6\alpha_A \tau_\infty} \frac{C_2^2}{C_1^2} \frac{t^2}{C_1 t}  + \|d\kappa\|_\infty^2 \frac{C_2^6}{6C_1^4} \frac{t^4}{C_1 t} n e^{2\alpha_A \tau_\infty} \right)\ dt.
    \end{align*}
    Therefore, \eqref{eq:WAest} holds with 
    \begin{align*}
	C = n^3 e^{6\alpha_A \tau_\infty} \frac{C_2^2}{C_1^3} \frac{\tau_\infty^2}{2}, \quad \text{and} \quad C' = n e^{2\alpha_A \tau_\infty} \frac{C_2^6}{C_1^5} \frac{\tau_\infty^4}{24},
    \end{align*}
    valid both for $W_A$ and $W_{A,\perp}$ as explained above. Theorem \ref{thm:estimate} is proved.
%    By Lemma \ref{lem:ODEK}, equation \eqref{eq:ODEK} allows us to derive the differential inequation, for $\ell = 1,2$
%    \begin{align*}
%	\frac{d}{dt} \|K_\ell\| (\cdot,t) \le \|(A + A^\star)/2\| (\varphi_t(\cdot)) \|K_\ell\| (\cdot,t) + |b_\ell(\cdot,t)| \|\star F_A\| (\varphi_t(\cdot)) \|E\| (\cdot,t).
%    \end{align*}
%    (One can reason on rows of $K_\ell$ separately and add the inequalities obtained over all rows). Using this and the estimate on $\|E|$ yields
%    \begin{align}
%	\|K_\ell\|(\x,v,t) \le \sqrt{n} e^{\alpha_A t} \int_0^t |b_\ell(\x,v,s)| \|\star F_A\|(\varphi_s(\x,v))\ ds, \qquad \ell = 1,2.  \label{eq:estK}
%    \end{align}    
\end{proof}

%%%%%%%%%%%%%%%%%%%%%%%%%%%%%%%%%%%%%%%%%%%%%%%%%%%%%%%% INJECTIVITY REDUCTION

\section{Injectivity equivalences and implications} \label{sec:reduction}

The purpose of this section is twofold. It first clarifies the relation between the transform $I_{A}$ restricted to one-forms, and the transform $I_{A,\perp}$. Second, it serves as preparation for the range characterization results stated in the next section. 

\subsection{On the range decomposition of $I_A$}
\label{subsec:decom}

We now prove that the range of $I_{A}$ acting on 1-forms (i.e. acting on $\Omega_{-1}\oplus\Omega_{1}$) decomposes into $(i)$ the range of $I_{A,\perp}$ defined on $H^1_0$ and $(ii)$ the ranges of $I_A$ restricted to $\ker^{\pm 1} \mu^*_{\pm} := \Omega_{\pm 1} \cap \ker \mu^*_{\pm}$ (or their $L^2$ versions), and that the sum vanishes if and only if all three components vanish.

The first thing to observe is the following lemma which follows right away form the ellipticity of $\mu_{\pm}$. 
\begin{lemma}\label{lem:decomp}
    Let $(M,g)$ be a Riemannian surface with boundary and $A$ a $C^1$ connection. The following decompositions hold, orthogonal for the $L^2(SM,\mathbb{C}^{n})$ inner product (hence unique): 
    \begin{itemize}
	\item[$(i)$] For every $f\in \Omega_1$, there exists $v\in \Omega_0$ with $v|_{\partial M} = 0$ and $g_1\in \ker^1 \mu_+^*$ such that $f = \mu_+ v + g_1$. 
	\item[$(ii)$] For every $f\in \Omega_{-1}$, there exists $v\in \Omega_0$ with $v|_{\partial M} = 0$ and $g_{-1}\in \ker^{-1} \mu_-^*$ such that $f = \mu_- v + g_{-1}$. 
    \end{itemize}
\end{lemma}

Lemma \ref{lem:decomp} implies that any one-form $\omega = \omega_1 + \omega_{-1}$ decomposes uniquely as follows: write $\omega_1 = \mu_+ a + g_1$ and $\omega_{-1} = \mu_- b + g_{-1}$, with $a,b \in H^1_0(M)$
and $g_{\pm 1}\in \ker^{\pm 1} \mu_{\pm}^*$. Upon defining $g_p := (a+b)/2$ and $g_s := (i(a-b)/2)$, the sum can be rewritten as 
\begin{equation}
    \omega_1 + \omega_{-1} = g_{-1} + (\mu_+ + \mu_-) g_p + \frac{\mu_+ - \mu_-}{i} g_s + g_1. 
    \label{eq:decom1}
\end{equation}
The transport equation 
\begin{align*}
    Xu + Au = - \omega_1 - \omega_{-1},
\end{align*}
can then be rewritten as 
\begin{align*}
    (X+A) (u + g_p) = - (g_{-1} + (X_\perp - A_V) g_s + g_1),
\end{align*}
where the functions $u$ and $u + g_p$ agree on $\partial SM$ so that 
\begin{align*}
    I_A [\omega_{-1} + \omega_1] = I_A g_{-1} + I_{A,\perp} g_s + I_A g_1.
\end{align*}

We will say that $I_{A}$ acting on 1-forms is {\it solenoidal injective} if whenever $I_{A}(\omega)=0$, there is smooth $p:M\to\mathbb{C}^{n}$ with $p|_{\partial M}=0$ such that $\omega=d_{A}p=dp+Ap$.
Lemma \ref{lem:decomp} implies the following: 
\begin{lemma}\label{lem:equivalence}
  For any $C^1$ connection $A$, $I_A$ is solenoidal injective on one-forms if and only if, for any $f\in C_0^1(M)$, $g_1\in \ker^{1} \mu_+^*$ and $g_{-1} \in \ker^{-1} \mu_-^*$, $I_A g_{-1} + I_{A,\perp} f + I_A g_1 = 0$ implies $f = g_1 = g_{-1} = 0$. 
\end{lemma}

\begin{proof} $(\implies)$ Suppose $I_A$ solenoidal injective and assume that $I_A g_{-1} + I_{A,\perp} f + I_A g_1 = 0$. Then from solenoidal injectivity, this means that there exists a function $h$ defined on $M$ vanishing on $\partial M$ such that 
  \begin{align*}
    (X+A) h = g_1 + (X_\perp - A_V) f + g_{-1}, 
  \end{align*}
  rewritten differently this means that $(\mu_+ + \mu_-)h = g_1 - i (\mu_+ - \mu_-) f + g_{-1}$, which upon projecting onto Fourier modes $1$ and $-1$, implies
  \begin{align*}
    \mu_+ (h+if) - g_1 = 0 = \mu_- (h-if) - g_{-1}. 
  \end{align*}
  Uniqueness of such decompositions implies $h+if = h-if = g_1 = g_{-1} = 0$, so $f = h =0$. \\
  \noindent $(\impliedby)$ Let $\omega$ be such that $I_A = 0$. Using Lemma \ref{lem:decomp}, we can write 
  \begin{align*}
    \omega = (X+A) g_p + (X_\perp - A_V) g_s + g_{-1} + g_{1},
  \end{align*}
  with $g_p, g_s$ functions on $M$ vanishing at $\partial M$ and $g_{\pm 1} \in \ker^{\pm 1} \mu_{\pm}^*$. Then 
  \begin{align*}
    0 = I_A \omega = I_A [(X_\perp - A_V) g_s + g_{-1} + g_{1}],
  \end{align*}
  which by assumption implies $g_s = g_{-1} = g_1 = 0$, thus $\omega = (X+A) g_p$. Hence $I_A$ is solenoidal injective on one-forms.   
\end{proof}

Given a $\mathbb{C}^{n}$-valued 1-form $\omega=\omega_{-1}+\omega_{1}$ there is an alternative decomposition
to \eqref{eq:decom1} which uses slightly different boundary conditions.
For this one considers the elliptic operator 
$$D: C^{\infty}_{0}(M,\mathbb{C}^{n})\times C^{\infty}(M,\mathbb{C}^{n})\to \Lambda^{1}(M)$$
where $\Lambda^{1}(M)$ is the set of all $\mathbb{C}^{n}$-valued 1-forms, given by
\[D(p,f)=d_{A}p+\star d_{A}f.\]
Now let $\mathfrak H_{A}$ denote the finite dimensional space of 1-forms $h$ such that $d_{A}h=d_{A}\star h=0$ and $j^*h=0$, where $j:\partial M\to M$ is the inclusion map. 
Using $D$ it is easy to show that given $\omega\in\Lambda^{1}(M)$ there are $(p,f)\in C^{\infty}_{0}(M,\mathbb{C}^{n})\times C^{\infty}(M,\mathbb{C}^{n})$ and $h\in \mathfrak{H}_{-A^{*}}$ such that
\begin{equation}
\omega=d_{A}p+\star d_{A}f+h.
\label{eq:decom2}
\end{equation}
Note that $\mathfrak{H}_{-A^{*}}$ is the ortho-complement to the range of $D$ (compare this with \cite[Lemma 6.1]{Paternain2013a}). Observe also that we can express \eqref{eq:decom2} as 
\[\omega_{-1}+\omega_{1}=(\mu_{+}+\mu_{-})p+\frac{\mu_{+}-\mu_{-}}{i}f+h_{1}+h_{-1}\]
where $h_{\pm 1}\in  \ker^{\pm 1} \mu_{\pm}^*$,
but the difference with \eqref{eq:decom1} is that now we do not require $f$ to vanish at the boundary and instead
we have $j^{*}h=0$. We will return to this alternative decomposition after proving Theorem \ref{thm:rangeCharac}.

\subsection{Injectivity for scalar perturbations of connections}

Given a connection $A$ on a simple surface $(M,g)$, we first start by giving a characterization of the injectivity for $I_{A,0}$. Recall that a function $f$ defined on $SM$ is so-called {\em (fiberwise) holomorphic} (resp. {\em antiholomorphic}) if $(Id + iH)f = f_0$ (resp. $(Id - iH)f = f_0$).

\begin{proposition}[Characterization of injectivity of $I_{A,0}$]
    Let $A$ be a $GL(n,\mathbb{C})$-connection. Then $I_{A,0}$ is injective if and only if the following is true: for any $f,u \in C^\infty (SM,\mathbb{C}^{n})$ satisfying $(X+A)u = -f$ with $u|_{\partial SM} =0$, 
  \begin{itemize}
    \item[$(i)$] If $f$ is holomorphic and even, then $u$ is holomorphic, odd. 
    \item[$(ii)$] If $f$ is antiholomorphic and even, then $u$ is antiholomorphic and odd. 
  \end{itemize}
  \label{prop:charac3}
\end{proposition}

\begin{proof}
    $(\implies)$ Suppose $I_{A,0}$ injective. We only prove $(i)$, as $(ii)$ is similar. Let $u, f$ as in the statement with $f$ holomorphic. Then $(Id - iH)f = f_0$. Moreover, projecting the transport equation onto odd harmonics, we obtain $(X+A)u_+ = 0$ with boundary condition $u_+|_{\partial SM} = 0$, hence $u_+ = 0$, thus $u$ is odd. We then compute
  \begin{align*}
    (X+A) (Id - iH)u &= (Id - iH) (X+A)u  -i [X+A,H] u \\
    &= - f_0 + \cancel{i (X_\perp - A_V) u_0} + i ((X_\perp - A_V)u)_0,
  \end{align*}
  which upon integrating along geodesics implies that $I_{A,0} [f_0 -i ((X_\perp - A_V)u)_0 ] = 0$. By assumption, this implies $f_0 + i((X_\perp - A_V)u)_0 = 0$. In particular, $(X+A)(Id-iH)u = 0$ with $(Id-iH)u|_{\partial SM} = 0$, hence $(Id-iH)u = 0$, which means that $u$ is holomorphic, hence the proof. \\
  $(\impliedby)$ Suppose $(i),(ii)$ are satisfied. Let $f$ be a smooth function such that $I_{A,0} f = 0$, then there exists $u:SM\to \Cm^n$ with $u|_{\partial SM} = 0$ and such that $(X+A)u = -f$. $f$ is even, both holomorphic and antiholomorphic, thus by $(i)$ and $(ii)$, $u$ is odd, both holomorphic and antiholomorphic, thus $u=0$, hence $f = 0$. Proposition \ref{prop:charac3} is proved. 
\end{proof}

The next result relies on the key concept of holomorphic integrating factor for scalar connections, which we now recall. Given a one-form $\omega$, there exists $v:SM\to \Cm$ holomorphic, even solution of $Xv = - \omega$. This is based on injectivity of the unattenuated transform $I_0$, cf. \cite[Theorem 4.1]{Paternain2012}. The construction goes as follows. First one may write $\omega = Xf + X_\perp g$ for $g$ vanishing at $\partial M$. Then we are left looking for $v$ such that $X(v+f) = - X_\perp g$. One can construct $u = (Id + iH) h_\psi$ with $h_\psi$ even such that 
\begin{align*}
  - X_\perp g = Xu = X(Id + iH) h_\psi = -i [H,X] h_\psi = - i X_\perp (h_\psi)_0. 
\end{align*}
By surjectivity of $I_0^*$, one can find $h$ such that $I_0^* h =  2\pi (h_\psi)_0 = -2\pi ig$ and for such an $h$, the function $v = -f + (Id + iH) h_\psi$ is a holomorphic, even solution of $Xv = -\omega$. As a result, the functions $e^v$ and $e^{-v}$ are non-vanishing holomorphic, even, solutions of $X e^{\pm v} \pm \omega e^{\pm v} = 0$. Using the same $h$, we can then construct $w = -f - (Id - iH)h_\psi$, anti-holomorphic solution of $Xw = -\omega$ giving rise to anti-holomorphic integrating factors $e^{\pm w}$ solutions of $X e^{\pm w} \pm \omega e^{\pm w} = 0$.  

With the use of such integrating factors, we are then able to establish the following.
\begin{proposition}\label{prop:transl2}
    For any $GL(n,\mathbb{C})$-connection $A$, if $I_{A,0}$ is injective, then for any smooth one-form $\omega$, so is $I_{A + \omega \Imm_n, 0}$.
\end{proposition}

\begin{proof} Suppose $I_{A,0}$ injective and let $\omega$ be a one-form. We use the characterization from Proposition \ref{prop:charac3} to show that $I_{A+\omega \Imm_n,0}$ is injective by satisfying $(i), (ii)$. Let $u,f$ be such that $(X+A +\omega)u = -f$ with $u|_{\partial SM} = 0$. If $f$ is holomorphic even, then $u$ is odd since $(X+A+\omega)u_+ = 0$ with zero boundary condition. Let $e^v$ a holomorphic, even, integrating factor for $\omega$, then we can recast $(X+A + \omega)u = -f$ as $(X+A) (e^{-v}u) = - e^{-v}f$, where $e^{-v} f$ is holomorphic, even and $e^{-v}u$ vanishes at $\partial SM$. Then since $A$ satisfies $(i)$, this implies that $e^{-v}u$ is holomorphic, odd, and hence $u = e^{v}(e^{-v}u)$ is holomorphic, odd. The proof of $(ii)$ is similar.     
\end{proof}

Such a result allows to derive injectivity results for several restrictions of $I_A$ to other subspaces of $C^\infty(SM)$, as they amount to studying transforms with connections which are translated from one another by a scalar one-form. Here and below, we denote $I_{A,k}$ the transform $I_A$ restricted to $\Omega_k$. 

\begin{proposition}\label{prop:shift}
    Suppose $I_{A,0}$ injective, then the following conclusions hold.
    \begin{itemize}
      \item[$(i)$] For any $k\in \Zm$, the transform $I_{A,k}$ is injective.
      \item[$(ii)$] $I_A$ is solenoidal injective over one-forms. In particular, $I_{A,\perp}$ is injective.  
    \end{itemize} 
\end{proposition}

\begin{remark}
   {\rm  In particular, both statements imply that $I_{A,k}|_{\ker^k \mu_+^*}$ and $I_{A,-k}|_{\ker^{-k} \mu_-^*}$ are both injective for every $k=0,1,2\dots$.  However this can be proved to always hold, see Proposition \ref{prop:jet} below.   }
\end{remark}

\begin{proof} Suppose $I_{A,0}$ injective. \\
  {\bf Proof of $(i)$.} Let $f\in \Omega_k$ such that $I_{A,k} f = 0$. Write $f = q^k \tilde{f}$ for $q$ a non-vanishing section of $\Omega_1$ and $\tilde{f}:M\to \Cm^n$. Then if $u$ is the unique solution to
  \begin{align*}
      (X+A)u = -q^k \tilde f\quad (SM), \qquad u|_{\partial_- SM} = 0, \qquad u|_{\partial_+ SM} = I_{A,k} f, 
  \end{align*}
  the function $q^{-k} u$ satisfies
  \begin{align*}
      (X + A + kq^{-1}Xq) (q^{-k}u) = -\tilde{f}\quad (SM), \qquad q^{-k}u|_{\partial_- SM} = 0,
  \end{align*}
  so that $(q^{-k} u)|_{\partial_+ SM} = I_{A+kq^{-1} Xq \Imm_n, 0} \tilde{f}$. In particular, this implies that 
  \begin{align*}
      I_{A+kq^{-1} Xq \Imm_n, 0} \tilde{f} = q^{-k}|_{\partial_+ SM} I_{A,k} f = 0.
  \end{align*}
  Since $I_{A+kq^{-1} Xq \Imm_n, 0}$ is injective by virtue of Proposition \ref{prop:transl2}, then $\tilde{f} = 0$, hence $f = 0$. \\
  {\bf Proof of $(ii)$.} Suppose $I_A (\omega_1 + \omega_{-1}) = 0$, then there exists $u$ such that $(X+A)u = - \omega_{-1} - \omega_1$ with $u|_{\partial SM} =0$. In particular, $u$ is even since $u_-$ is a first integral of $X+A$ vanishing at $\partial SM$. If $q\in \Omega_1$ is non-vanishing, the equation $(X+A)u = -\omega_{-1} - \omega_1$ can be rewritten as 
  \[ (X+A - q^{-1} Xq) (qu) = - q(\omega_{-1} + \omega_1).  \]
  If $e^v$ is a holomorphic, even, solution of $X e^v - q^{-1}Xq e^v = 0$, then this equation can be rewritten as
  \[ (X+A) (e^{-v} qu) = - e^{-v}q(\omega_{-1} + \omega_1), \qquad (e^{-v}qu)|_{\partial SM} = 0,    \]
  and since the right hand side is holomorphic and even, then by injectivity of $I_{A,0}$, $e^{-v}qu$ is holomorphic and odd. Then $u = q^{-1} e^v (e^{-v} qu)$ has harmonic content no less than $-1$ and since $u$ is even, $u_{-1} = 0$ as well, so $u$ is holomorphic. Similarly using an antiholomorphic integrating factor, one may show that $u$ is antiholomorphic, so we conclude that $u = u_0$ with $u_0|_{\partial M} = (u|_{\partial SM})_0 = 0$, and the relation $(X+A)u_0 = - \omega_1 - \omega_{-1}$ implies that $I_{A}$ is solenoidal injective over one-forms. The proof is complete. 
\end{proof}

Finally, the next two propositions aim at showing that $I_{A,\perp}$ injective implies that $I_{A,0}$ injective. 

\begin{proposition}[Characterization of injectivity of $I_{A,\perp}$]
    Let $A$ be a smooth $GL(n)$ connection. Then $I_{A,\perp}:C_0^\infty(M)\to C^\infty(\partial_+ (SM))$ is injective if and only the following is true: for any $f,u \in C^\infty (SM)$ satisfying $(X+A)u = -f$ with $u|_{\partial SM} =0$, $f$ odd and $u$ even, 
  \begin{itemize}
      \item[$(i)$] If $f_k = 0$ for all $k< -1$ and $f_{-1}\perp \ker^{-1} \mu_-^*$, then $u$ is holomorphic. 
    \item[$(ii)$] If $f_k = 0$ for all $k> 1$ and $f_1 \perp \ker^1 \mu_+^*$, then $u$ is antiholomorphic. 
  \end{itemize}
  \label{prop:characIperp}
\end{proposition} 

\begin{proof}
    $(\implies)$ Suppose $I_{A,\perp}$ injective. We only prove $(i)$ as $(ii)$ is similar. Let $u,f$ as in the statement with $f_k = 0$ for all $k< -1$ and $f_{-1} \perp \ker^{-1} \mu_-^*$. In particular, from Lemma \ref{lem:decomp}, we can write $f_{-1} = \mu_- v_0$ with $v_0|_{\partial M} = 0$. Then $(Id - iH)f = 2f_{-1} = 2\mu_- v_0$. The function $(Id - iH) u$ solves 
  \begin{align*}
    (X+A) (Id - iH)u &= (Id - iH) (X+A)u  -i [X+A,H] u \\
    &= - 2\mu_- v_0 + i (X_\perp - A_V) u_0 + \cancel{i ((X_\perp - A_V)u)_0}.
  \end{align*}
  With $2\mu_- v_0 = (X+A) v_0 - i (X_\perp -A_V) v_0$, the equation above becomes: 
  \begin{align*}
      (X+A) ((Id - iH)u + v_0) =  i (X_\perp - A_V) (u_0 + v_0). 
  \end{align*}
  Upon integrating along geodesics, we get 
  \begin{align*}
      I_{A,\perp} (i(u_0 + v_0)) = -B_{A,-} ((Id - iH)u + v_0)|_{\partial SM} = 0,
  \end{align*}
  which by injectivity of $I_{A,\perp}$ implies $u_0 + v_0 = 0$. Then the transport equation above becomes 
  \[  (X+A) [(Id - iH)u + v_0] = 0, \qquad  ((Id - iH)u + v_0)|_{\partial SM} = 0,  \]
  which implies $(Id - iH) u = -v_0$, thus $u$ is holomorphic. \\
  \noindent $(\impliedby)$ Suppose $(i),(ii)$ satisfied. Let $h\in C_0^\infty(M)$ such that $I_{A,\perp} h = 0$, then there exists $u:SM\to \Cm^n$ with $u|_{\partial SM} = 0$ and such that $(X+A)u = -(X_\perp - A_V) h$. Then $u$ is even since $(X+A) u_- = 0$ with $u_-|_{\partial SM} = 0$. Then $f = (X_\perp - A_V) h = i (\mu_- h- \mu_+ h)$ satisfies requirements for both $(i)$ and $(ii)$, so that $u$ is both holomorphic and anti-holomorphic. Then $u= u_0$ with $u_0|_{\partial M} = 0$. Then the relation $(X+A)u_0 = - (X_\perp - A_V) h$ implies $\mu_+ (iu_0 + h) = 0$ and $\mu_- (iu_0 - h) = 0$. Since $(iu_0\pm h)|_{ \partial M} = 0$, this implies $iu_0 \pm h = 0$, hence $u_0 = h = 0$, and $I_{A,\perp}$ is injective. Proposition \ref{prop:characIperp} is proved. 
\end{proof}

\begin{proposition} 
    Let $A$ be a $GL(n,\mathbb{C})$ connection. If $I_{A,\perp}$ is injective, then so is $I_{A,0}$.     
    \label{prop:IperpI0}
\end{proposition}

\begin{proof} Suppose $I_{A,\perp}$ injective so that it satisfies $(i)$ and $(ii)$ in Proposition \ref{prop:characIperp}. Let $f$ such that $I_{A,0}f = 0$. Then there exists $u$ odd such that $(X+A)u = -f$ with $u|_{\partial SM} =0$. With $q$ a non-vanishing section of $\Omega_1$ and $\omega:= -q^{-1} Xq$, this implies
    \[ (X+A + \omega \Imm_n ) (qu) = - qf.  \]
    Let $e^w$ a holomorphic, even function such that $X e^w + \omega e^w = 0$, then the equation above can be rewritten as 
    \[ (X+A) (que^{-w}) = -qfe^{-w}, \qquad que^{-w}|_{\partial SM} = 0,   \]
    where $qfe^{-w}$ is odd and $que^{-w}$ is even. Moreover, $qfe^{-w}$ is holomorphic, thus satisfies the requirement for $(i)$, hence $que^{-w}$ is holomorphic, hence $u = u_{-1} + u_1 + u_3 \dots$. Using a similar argument with $(ii)$, we can then cancel all $u_k$'s for $k\ge 2$. Thus $u = u_{-1} + u_1$. Projecting the equation $(X+A)u = -f$ onto $\Omega_2$ and $\Omega_{-2}$ gives $\mu_+ u_1 = \mu_- u_{-1} = 0$, and since $u_1|_{\partial SM}  = u_{-1} |_{\partial SM} =0$, this implies $u_1 = u_{-1} = 0$, hence $f = 0$.
\end{proof}

We conclude by proving the following result which has independent interest.

\begin{proposition}Suppose there is $u\in\Omega_{k}$ such that $I_{A,k}(u)=0$. Then $u$ has vanishing jet at $\partial M$.
In particular $I_{A,k}$ is injective when restricted to $\text{\rm Ker}\,\mu_{\pm}$.
\label{prop:jet}
\end{proposition} 

\begin{proof} The main observation is that $N=I_{A,k}^{*} I_{A,k}$ is an elliptic classical $\Psi DO$ of order $-1$ in the interior of any simple manifold engulfing $M$, see \cite[Section 5]{Paternain2013a} and references therein. Hence consider a slightly larger simple manifold $M_{1}$ containing $M$ and extend $u$ by zero to $M_{1}$ ($A$ is extended in any smooth way). Thus $Nu=0$ in the interior of $M_{1}$ and by elliptic regularity we deduce that $u$ is smooth in $M_{1}$. Since $u$ vanishes outside $M$, this clearly imply that $u$ has zero jet at the boundary of $M$.

Suppose in addition $\mu_{-}(u)=0$. If we write $u=he^{ik\theta}$ then using (\ref{eq:mu}) we see that $\bar{\partial}(he^{k\lambda})+A_{\bar{z}}he^{k\lambda}=0$. Using the existence of $F:M\to GL(n,\mathbb{C})$ such that $\bar{\partial} F+A_{\bar{z}}F=0$ as in Lemma \ref{lemma:solveCR} we see that $\bar{\partial}(F^{-1}he^{k\lambda})=0$. Since $h$ vanishes on $\partial M$, this is enough to conclude
that $u=0$. A similar argument applies to elements in the kernel of $\mu_{+}$ (or their adjoints). 
\end{proof}

%%%%%%%%%%%%%%%%%%%%%%%%%%%%%%%%%%%%%%%%%%%%%%%%%%%%%%%% RANGE CHARACTERIZATION

\section{Range characterization} \label{sec:range}

We start with a standard surjectivity result. 

%Recall that throughout this section, we assume hypothesis \ref{hyp:injectivity} to hold, that is, $I_{A,0}$ is injective for all $GL(n,\mathbb{C})$ connections $A$. We will make use of of the following result.

\begin{theorem}\label{thm:surjective}
    Suppose $I_{A,0}$ is injective. Given $f\in C^{\infty}(M,\Cm^n)$ there exists $h\in {\mathcal S}_{-A^{*}}^{\infty}(\partial_{+}(SM),\Cm^n)$ such that $I_{A,0}^*(h)=f$. 
\end{theorem}

The proof of this result is now well-understood and we omit it. It follows from injectivity of $I_{A,0}$ and the fact that $I^*_{A,0} I_{A,0}$ is an elliptic classical $\Psi DO$ of order $-1$ in the interior of any simple manifold engulfing $M$, see \cite[Section 5]{Paternain2013a} and references therein. \\
From the expression $I_{A,0}^* h = 2\pi (h_{-A^*,\psi})_0$, upon setting $u = 2\pi\ h_{-A^*,\psi}\in C^\infty(SM,\Cm^n)$, Theorem \ref{thm:surjective} is equivalent to stating that for every $f\in C^{\infty}(M,\Cm^n)$, there exists $u\in C^\infty(SM,\Cm^n)$ satisfying $(X-A^*)u = 0$ and $u_0 = f$.

The next result is less standard and it is based on the solvability result given by Lemma \ref{lemma:solveCR} and follows the strategy of the proof of \cite[Theorem 5.5]{Paternain2013a}.

\begin{theorem} Suppose $I_{A,0}$ is injective. Given $f\in C^{\infty}(M,\Cm^n)$ there exists $h\in {\mathcal S}_{-A^{*}}^{\infty}(\partial_{+}(SM),\Cm^n)$ such that
$I_{A,\perp}^*(h)=f$. \label{thm:surjectiveperp}
\end{theorem}

\begin{proof} Consider the purely imaginary 1-form
    \begin{align}
	a:=A_{\xi,g}= -q^{-1} Xq, \qquad\qquad a = -\bar{a}.	
	\label{eq:a}
    \end{align}
    where $q\in\Omega_1$ is nowhere vanishing (e.g. in global isothermal coordinates $q=e^{i\theta}$). Observe that if $u:SM\to\Cm^n$ is any smooth function then
    \begin{equation}
	(X-A^*-ma \Imm_n)u=q^{-m}((X-A^*)(q^m u))
	\label{eq:att-t}
    \end{equation}
    where $m\in \Zm$. First we show the following result which is interesting in its own right:
    
    \begin{lemma}\label{lemma:sobre} Suppose $I_{A,0}$ injective. Given any $f\in\Omega_m$, there exists $w\in C^{\infty}(SM,\Cm^n)$ such that
	\begin{enumerate}
	    \item $(X-A^*)w=0$,
	    \item $w_m=f$.
	\end{enumerate}	
    \end{lemma}
    
    \begin{proof} Since $I_{A,0}$ is injective, by Proposition \ref{prop:transl2}, $I_{A-ma \Imm_n,0}$ is injective (with $a$ defined in \eqref{eq:a}), thus by Theorem \ref{thm:surjective}, there is $u\in C^{\infty}(SM,\Cm^n)$ such that $0 = (X-A^* +m \bar{a} \Imm_n)u= (X-A^* -m a \Imm_n)u$ and $u_0=q^{-m}f$. If we let $w:=q^{m}u$, then clearly
	$w_m=f$ and by (\ref{eq:att-t}) we also have $(X-A^*)w=0$.
    \end{proof}
    
    As before, consider the operators $\mu_{\pm}=\eta_{\pm}^{A}=\eta_{\pm}+A_{\pm 1}$. Clearly, 
    \[  X+A=\mu_{+}+\mu_{-}, \qquad X-A^* =\eta_{+}^{-A^{*}}+\eta_{-}^{-A^{*}}= -\mu_+^* - \mu_-^* \quad\text{and}\quad X_\perp - A_V = \frac{\mu_+ - \mu_-}{i}. \]
    We need the following solvability result which is a direct consequence of Lemma \ref{lemma:solveCR}.

    \begin{lemma}\label{lemma:solvability} Given $f\in C^{\infty}(M,\Cm^n)$ there are $w_1\in\Omega_1$ and $w_{-1}\in\Omega_{-1}$ such that
	\begin{align}
	    &\eta_{+}^{-A^{*}}(w_{-1})+\eta_{-}^{-A^{*}}(w_{1})=0,\label{eq:1}\\
	    &\eta^{-A^{*}}_{+}(w_{-1})-\eta_{-}^{-A^{*}}(w_{1})=f/(2 \pi i) .\label{eq:2}
	\end{align}	
    \end{lemma}
    
    \begin{proof} Obviously the claim is equivalent to showing that there exists $w_1\in\Omega_1$ such that
$\eta_{-}^{-A^{*}}(w_{1})=-f/4\pi i$ and $w_{-1}\in\Omega_{-1}$ such that $\eta_{+}^{-A^{*}}(w_{1})=f/4\pi i$.
This follows directly from Lemma  \ref{lemma:solveCR}.
\end{proof}

%We seek for solutions of the form $w_{1}=\mu_{+}a$ and $w_{-1}=\mu_{-}b$, where $a,b\in C^{\infty}(M,\Cm^n)$. This turns (\ref{eq:1}), %(\ref{eq:2}) into a second order system whose second order leading term is
%	\[ -\eta_{-}\eta_{+}\left[ \begin{smallmatrix} 1 & 1 \\ -1 & 1 \end{smallmatrix} \right]. \]
%	Observe that for $a\in \Omega_{0}$, $\eta_{-}\eta_{+}a=\eta_{+}\eta_{-}a$ and up to a constant factor this is precisely the Laplacian $\Delta$ %of $(M,g)$. Hence, after performing a simple linear transformation, we have reduced the lemma to a solvability result for an equation of the form
%	\[(\Delta\, \Imm_{2n}+Q)u=g,\]
%	where $Q$ is a first order differential operator with smooth coefficients. This is exactly the situation considered in \cite[Lemma 6.5]

%{Paternain2013a}. The result also follows from semiglobal solvability results \cite[Section 26.1 and references therein]{Hor}.

    We are now in good shape to complete the proof of Theorem \ref{thm:surjectiveperp}. Given $f\in C^{\infty}(M,\Cm^n)$, we consider the functions $w_{\pm 1}\in\Omega_{\pm 1}$ given by Lemma \ref{lemma:solvability}.
   By Lemma \ref{lemma:sobre} we can find odd functions $p,q\in C^{\infty}(SM,\Cm^n)$ solving the transport equation $(X-A^*)p=(X-A^*)q=0$ and with $p_{-1}=w_{-1}$ and $q_{1}=w_{1}$. Then the smooth function
    \[w:=\sum_{-\infty}^{-1}p_k+\sum_{1}^{\infty}q_k\]
    satisfies $(X-A^*)w=0$ thanks to equation (\ref{eq:1}). Upon defining $h = w|_{\partial_{+}SM}$ so that $w = h_{\psi,-A^*}$, we then obtain that $h$ satsfies
    \begin{align*}
	I_{A,\perp}^* h = -2\pi \pi_0 (X_\perp + A^{*}_V) h_{\psi,-A^*} &= 2\pi i\ \pi_0 (\eta_{+}^{-A^{*}} - \eta_{-}^{-A^{*}}) w\\ &= 2\pi i\ (\eta_{+}^{-A^{*}} (w_{-1}) - \eta_{-}^{-A^{*}} (w_1)) \stackrel{\eqref{eq:2}}{=} f, 
    \end{align*}
    as desired. 
\end{proof}
    %Solving $\mu_{-}w_{1}=f$ essentially reduces to solving an equation like $\bar{\partial}_{A}g=f$ in a disk
    %which I believe can always be done by reducing it to $\bar{\partial}g=f$ using that there is a smooth
    %map $F:M\to GL(n,\mathbb{C})$ such that $F^{-1}\bar{\partial}(F\cdot )=\bar{\partial}_{A}(\cdot)$.
    %Hence in the unitary case there may be an easier proof than the one found in \cite[Lemma 6.1]{Paternain2013a}.
    %I recall that that lemma was one of the trickiest bit of the paper. But I'm hoping we can now use that proof
    %for the non-unitary case. As far as I can see it works, but we need some careful checking.
    
    %The idea is as follows. Set $w_{1}=\mu_{+}g$ and $w_{-1}=\mu_{-}h$ and substitute into (\ref{eq:1}) and %(\ref{eq:2}) to obtain a second order system in the unknowns $g$ and $h$.
    %The key observation is that after a simple transformation, the principal part of the system is diagonal and like a regular laplacian
    %with a bunch of complicated terms of order 1 and 0 that should not matter.
    %So now we want to use \cite[Lemma 6.5]{Paternain2013a} and complete the proof.

Finally, with the surjectivity Theorems \ref{thm:surjective} and \ref{thm:surjectiveperp}, we are now ready to prove Theorem \ref{thm:rangeCharac}. As explained in the Introduction, define $P_A:{\mathcal S}_A^{\infty}(\partial_{+}(SM),\Cm^n) \to  C^{\infty}(M,\Cm^n)$, as follows
\[  P_A := B_{A,-} H Q_{A,+}.    \]
The operator $P_A$ is a boundary operator which only depends on the scattering relation and the scattering data $C_A$. Upon splitting the Hilbert transform $H$ into its projections onto even and odd harmonics (call them $H_+$ and $H_-$), we obtain the splitting $P_{A} = P_{A,+} + P_{A,-}$, where we have defined $P_{A,\pm}:= B_{A,-} H_\pm Q_{A,+}$.

\begin{proof}[Proof of Theorem \ref{thm:rangeCharac}]
    For $w$ defined on $\partial_+ (SM)$, recall that $Q_{A,+} w = w_{\psi,A}|_{\partial SM}$ and that \\ \hbox{$B_{A,-}(u|_{\partial(SM)}) = I_A( -(X+A)u )$}. Using these considerations and the commutator formulas, we are able to derive
    \begin{align*}
	P_{A,+} w = B_{A,-}H_+ Q_{A,-}w = B_{A,-} (H_+ w_{\psi,A})|_{\partial (SM)} &= I_A (-(X+A) H_+ w_{\psi,A}) \\
	&= I_A (  (H_-(X + A) - (X+A)H_+) w_{\psi,A} ) \\
	&= I_A ( (X_\perp - A_V)\pi_0 w_{\psi,A} ) \\
	&= \frac{1}{2\pi} I_{A,\perp} I^*_{\Aadj,0} w. 
    \end{align*}
    Similarly for $P_{A,-}$,
    \begin{align*}
	P_{A,-} w = B_{A,-}H_- Q_{A,-}w = B_{A,-} (H_- w_{\psi,A})|_{\partial (SM)} &= I_A (-(X+A) H_- w_{\psi,A}) \\
	&= I_A (  (H_+ (X+A) - (X+A) H_-) w_{\psi,A} ) \\
	&= I_A ( \pi_0 (X_\perp - A_V) w_{\psi,A}  ) \\
	&= - \frac{1}{2\pi} I_{A,0} I^*_{\Aadj,\perp} w. 
    \end{align*}
    Since it is assumed that $I_{-A^*,0}$ is injective, by virtue of Theorems \ref{thm:surjective} and \ref{thm:surjectiveperp}, the operators $I_{\Aadj,0}^*,I_{\Aadj,\perp}^*:{\mathcal S}_A^{\infty}(\partial_{+}(SM),\Cm^n) \to  C^{\infty}(M,\Cm^n)$ are surjective. Combining this surjectivity with the two factorizations above, claims $(i)$ and $(ii)$ follow.  
\end{proof}

\begin{remark} {\rm Examining the proof above, we then see that 
    \begin{align*}
	P_A = \frac{1}{2\pi} (I_{A,\perp} I^*_{\Aadj,0} - I_{A,0} I^*_{\Aadj,\perp}).
    \end{align*}
    In addition, on the direct sum 
    \begin{align*}
	{\mathcal S}_A^{\infty}(\partial_{+}(SM),\Cm^n) = \V_{A,+} \oplus \V_{A,-},
    \end{align*}
    since $I^*_{\Aadj,0}$ vanishes on $\V_{A,-}$ and $I^*_{\Aadj,\perp}$ vanishes on $\V_{A,+}$, one realizes that $P_{A,\pm}$ coincides with the restriction $P_A|_{\V_{A,\pm}}$. This is also true since, following previous observations, if $h\in \V_{A,+}$, then $Q_{A,+} h$ is even and if $h\in \V_{A,-}$, then $Q_{A,+}h$ is odd, which justifies the corresponding splitting of the Hilbert transform into odd and even parts in the previous definitions.  }  
\end{remark}

\subsection{Comparison with the range characterization in \cite{Paternain2013a}}
We conclude this section by making a comparison between Theorem \ref{thm:rangeCharac} and the range characterization of $I_{A}$ acting on 1-forms in \cite[Theorem 1.3]{Paternain2013a} when $A$ is skew-hermitian.
The first thing to observe is that due to our sign conventions $P_{A,+}$ is precisely $-P_{+}$ in \cite{Paternain2013a}, so the main difference is the presence of $I_{A}(\mathfrak{H}_{A})$, where
$\mathfrak{H}_{A}$ was introduced in Subsection \ref{subsec:decom}.
%Here $\mathfrak H_{A}$ denotes the finite dimensional space of 1-forms $\omega$ such that %$d_{A}\omega=d^*_{A}\omega=0$ and $j^*\omega=0$, where $j:\partial M\to M$ is the inclusion map. This %space is trivial if $A=0$, but it could in principle be non-zero for general $A$. It is natural to call these forms $A$-%harmonic.
The reason why $\mathfrak{H}_{A}$ does not appear in Theorem \ref{thm:rangeCharac} is that we are only considering the range of $I_{A,\perp}$. In fact for a general $GL(n,\mathbb{C})$-connection $A$ we have:

\begin{lemma} Assume $I_{A}$ is solenoidal injective on 1-forms. Then
{\rm \[ \text{range}\;I_{A}= \text{range}\;I_{A,\perp}\oplus I_{A}(\mathfrak{H}_{-A^{*}}).\]}
\end{lemma}

\begin{proof} The fact that the range splits follows directly from the decomposition \eqref{eq:decom2} and the definitions. The sum is direct because of the following observation: if $h\in \mathfrak{H}_{-A^{*}}$ is such that $I_{A}(h)\in \text{range}\;I_{A,\perp}$ then $h=0$. Indeed, in this case there is $f\in C^{\infty}(M,\mathbb{C}^{n})$
such that $I_{A}(\star d_{A}f+h)=0$. Since $I_{A}$ is solenoidal injective, we have that there is $p\in C^{\infty}(M,\mathbb{C}^{n})$ with $p|_{\partial M}=0$
such that $d_{A}p=\star d_{A}f+h$. This implies right away that $h=0$.

\end{proof}

We conclude with an example showing that $\mathfrak{H}_{A}$ could be non-trivial. We note that
 $\mathfrak{H}_{A}$ transforms isomorphically under gauge equivalences and it is trivial for $A=0$ (hence it is zero for any flat connection).
For the example, suppose $M$ is the unit disk with the standard metric. Consider the following map $F:\partial M={\mathbb S}^{1}\to SU(2)$ given by
\begin{equation}
F(e^{i\phi})=\left[ \begin{matrix} e^{-2i\phi} & 0 \\ 0 & e^{2i\phi} \end{matrix} \right].
\label{eq:boundarycon}
\end{equation}
Since $SU(2)$ is simply connected $F$ can be extended to a smooth map $F:M\to SU(2)$.
Define the $GL(2,\mathbb{C})$-connection $A:=-(\bar{\partial} F)F^{-1}d\bar{z}=A_{\bar{z}}d\bar{z}$.
Thus
\begin{equation}
\bar{\partial}F+A_{\bar{z}}F=0,\label{eq:bar}.
\end{equation}
We claim that there is a non-zero 1-form $h$ such that $d_{A}h=d_{A}\star h=0$ and $j^* h=0$. Indeed, let $h:=h_{\bar{z}}d\bar{z}+h_{z}dz$ where
\begin{align*}
    h_{\bar{z}}(x,y):=\left[ \begin{matrix} 1 \\ 0 \end{matrix} \right], \qquad 
    h_{z}(x,y):=F(x,y)\left[ \begin{matrix} 1 \\ 0 \end{matrix} \right].    
\end{align*}
Using \cite[Lemma 6.2]{Paternain2013a} (which holds for all $A$, not just unitary ones) we see that $\star d_{A}h=2i(\mu_{-}(h_{1})-\mu_{+}(h_{-1}))$ and thus
$d_{A}h=0$ and $d_{A}\star h=0$ are equivalent to $\mu_{-}(h_{1})=\mu_{+}(h_{-1})=0$.
These equations hold because of (\ref{eq:bar}) and $A_{z}=0$. Finally the boundary condition $j^*h=0$ holds because
using (\ref{eq:boundarycon}):
\[h(ie^{i\phi})=\left[ \begin{matrix} 1 \\ 0 \end{matrix} \right](-i)e^{-i\phi}+F(e^{i\phi})\left[ \begin{matrix} 1 \\ 0 \end{matrix} \right]ie^{i\phi}=0.\]

It is now natural to ask: is there a way to characterize the finite dimensional subspace $I_{A}(\mathfrak{H}_{-A^{*}})$ in terms of boundary data?

%Further comments:
%\begin{itemize}
%  \item Regardless of whether these transforms are injective, the fact that the reconstruction equations are Fredholm implies that elements in the kernel of this transform are smooth. Pseudodifferential theory already proves that, although in our case, the remainder operator admits a rather explicit expression.
%  \item If injectivity of $I_{A,0}$ was enough to show injectivity of $Id + W_A^2$, then this would imply injectivity of $Id + W_{\Aadj,\perp}^2$, then injectivity of $I_{\Aadj, \perp}$ and vice versa. The main problem is the same as usual, even in the geodesic, scalar case: does $I_0$ injective imply $Id + W^2$ injective ?
%  \item Injectivity of transforms with connections are tied to the injectivity of the corresponding transform with the adjoint connection. This tie becomes trivial when the connection is hermitian but a symmetry breaks once the connection is not hermitian and we have to consider more equations jointly. Maybe write a section on Hermitian connections at the end ? 
%\end{itemize}

%%%%%%%%%%%%%%%%%%%%%%%%%%%%%%%%%%%%%%%%%%%%%%%%% NUMERICS

%\section{Numerical examples} \label{sec:numerics}

%\input numerics

\section*{Acknowledgements.} The authors thank the referees for their constructive and
useful comments. FM was partially funded by NSF grant DMS-1712790 and GPP by EPSRC grant EP/M023842/1.

\appendix

\section{On the simplicity constants $C_1$, $C_2$}

%%%%%% connect appropriately
Recalling that we introduced constants $C_{1,2}(M,g)$ in \eqref{eq:simpleclaim} for simple surfaces, let us mention that these constants cannot be made universal. Indeed, consider a centered disk of radius $1-\varepsilon$ in $\Rm^2$ with a constant curvature metric $+1$ or $-1$. For every $\varepsilon>0$, these surfaces are simple, though in the first case, $C_1$ tends to zero as $\varepsilon\to 0$ and in the second case, $C_2$ grows unboundedly as $\varepsilon\to 0$. It is unclear how to sharply characterize simplicity via these constants in terms of intrinsic quantities of the surface. A sufficient way to achieve this can be done as follows. For any $s\in \Rm$, let us denote $s^{\pm} = \max (\pm s, 0)$, and for a non-trapping surface $(M,g)$, let us define
\begin{align*}
  k^{\pm} (M,g) = \sup_{(\x,v)\in \partial_+ (SM)} \int_0^{\tau_+(\x,v)} t \kappa^{\pm} (\varphi_t(\x,v))\ dt. 
\end{align*}
The constant $k^+(M,g)$ appears in \cite[pp119-120]{Sharafudtinov1994}, chosen to be a dimensionless quantity (in that it does not vary under multiplication of $g$ by a positive number) and the condition $k^+(M,g)<1$ is a sufficient (but not necessary) condition for the absence of conjugate points. We have the following bounds on the function $b$ in terms of $k^{\pm}$.
\begin{lemma} \label{lem:boundsb2}
  Let $(M,g)$ a non-trapping Riemannian surface such that $k^+(M,g)<1$. Then the function $b_2$ satisfies
  \begin{align*}
    1-k^+ (M,g) \le \frac{|b_2(\x,v,t)|}{t} \le \exp (k^- (M,g)), \qquad (\x,v)\in SM, \qquad t\in [0,\tau(\x,v)].
  \end{align*}
\end{lemma}

\begin{proof} Denote for brevity $b(\x,v,t) := -b_2(\x,v,t)$, the unique solution to 
    \begin{align*}
	\ddot b + \kappa(\varphi_t) b = 0, \qquad b (0) = 0, \qquad \dot b (0) = 1,
    \end{align*}
    We have for every $\x\in M$, the estimate $-\kappa^-(\x)\le \kappa(\x) \le \kappa^+(\x)$. Therefore, by Sturm comparison, the pointwise estimate holds
    \begin{align*}
	\underline{b}(\x,v,t)\le b(\x,v,t) \le \overline{b}(\x,v,t), \qquad (\x,v)\in SM, \qquad t\in [0,\tau(\x,v)],
    \end{align*}
    where $\underline{b}$ and $\overline{b}$ solve the following problems:
    \begin{align*}
	\ddot{\underline{b}} + \kappa^+ (\varphi_t(\x,v)) \underline{b} = 0, \qquad \underline{b}(0) = 0, \qquad \underline{\dot b}(0) = 1, \\
	\ddot{\overline{b}} - \kappa^- (\varphi_t(\x,v)) \overline{b} = 0, \qquad \overline{b}(0) = 0, \qquad \overline{\dot b}(0) = 1.
    \end{align*}
    $\underline{b}$ has a simple zero at $t=0$ so we can consider the smooth function $\frac{\underline{b}(\x,v,t)}{t}$, and since $k^+(M,g)<1$, $\underline{b}$ remains positive for $t>0$. Direct integration of the previous ODE gives us that 
    \begin{align*}
	\frac{\underline{b}(t)}{t} &= 1 - \int_0^t \frac{t-u}{t} u \kappa^+(\varphi_u) \frac{\underline{b}(u)}{u}\ du \\
	&= 1 - \int_0^t \frac{t-u}{t} u \kappa^+(\varphi_u)\ du + \int_0^t \frac{t-u}{t} u \kappa^+(\varphi_u) \int_0^u \frac{u-v}{u} v \kappa^+ (\varphi_v) \frac{\underline{b}(v)}{v}\ dv\ du \\
	&\ge 1 - \int_0^t u \kappa^+(\varphi_u)\ du = 1- k^+(M,g). 
    \end{align*}
    On to the upper bound, we integrate the ODE satisfied by $\overline{b}$ in a similar fashion, and obtain
    \begin{align*}
	\frac{\overline{b}(t)}{t} &= 1 + \int_0^t \frac{t-u}{t} u \kappa^-(\varphi_u) \frac{\overline{b}(u)}{u}\ du \le 1 + \int_0^t u\kappa^-(\varphi_u) \frac{\overline{b}(u)}{u}\ du, 
    \end{align*}
    so that, by Gr\"onwall's inequality, we obtain $\frac{\overline{b}(t)}{t} \le \exp\left( \int_0^t u \kappa^-(\varphi_u)\ du \right)$. Thus we can bound uniformly 
    \begin{align*}
	\frac{b(\x,v,t)}{t} \le \frac{\overline{b}(\x,v,t)}{t} \le \exp\left( k^-(M,g) \right).
    \end{align*}
    The proof is complete. 
\end{proof}

%\begin{remark}
%    A much cruder bound could also be obtained by just bounding 
%    \begin{align*}
%	\kappa(x)\ge - \kappa_0 = - \max_{\x\in M} (-\kappa(x)),
%    \end{align*}
%    which would bound $b$ above by the solution to $\ddot c - \kappa_0 c = 0$  with $c(0) = 0$ and $\dot c(0) = 1$. We have $c(t) = \frac{1}{\sqrt{\kappa_0}}\sinh (\sqrt{\kappa_0}t)$, which implies
%    \begin{align*}
%	\frac{b(t)}{t} \le \frac{\sinh (\sqrt{\kappa_0}t)}{\sqrt{\kappa_0}t}. 
%    \end{align*}    
%\end{remark}

We now prove Lemma \ref{lem:Wkernelbound}, used in the proof of Theorem \ref{thm:estimate}.

\begin{proof}[Proof of Lemma \ref{lem:Wkernelbound}]
    Let us define $\Phi(\x,v,t) = \left[ \begin{smallmatrix} b_1 & b_2 \\ c_1 & c_2 \end{smallmatrix} \right] (\x,v,t)$. This function satisfies 
    \begin{align*}
	\frac{d}{dt} \Phi + \Km (\varphi_t(\x,v)) \Phi = 0, \qquad \Phi(\x,v,0) = \Imm_2,
    \end{align*}
    where we have defined $\Km := \left[ \begin{smallmatrix} 0 & 1 \\ -\kappa & 0 \end{smallmatrix} \right]$. By Wronskian constancy, since $\Km$ is traceless then $\det \Phi (\x,v,t) \equiv 1$. The function $V\Phi$ solves the ODE
    \begin{align*}
	\frac{d}{dt} (V\Phi) + \Km(\varphi_t) V\Phi = - (V(\Km\circ \varphi_t)) \Phi, \qquad V\Phi(\x,v,0) = 0, 
    \end{align*}
    where we have $- V(\Km \circ \varphi_t) = b_2(\x,v,t)\kappa_\perp(\varphi_t(\x,v)) \left[ \begin{smallmatrix} 0 & 0 \\ 1 & 0 \end{smallmatrix} \right]$ with $\kappa_\perp = X_\perp \kappa$. Using $\Phi$ itself as an integrating factor of the latter equation, we can arrive at the integral formula (keeping $(\x,v)$ implicit)
    \begin{align*}
	V\Phi(\x,v,t) = \Phi(t) \int_0^t b_2(s) \kappa_\perp(\varphi_s) \Phi^{-1} (s) \left[ \begin{smallmatrix} 0 & 0 \\ 1 & 0 \end{smallmatrix} \right] \Phi(s)\ ds.
    \end{align*}
    We compute $\Phi^{-1} (s) \left[ \begin{smallmatrix} 0 & 0 \\ 1 & 0 \end{smallmatrix} \right] \Phi(s) = \left[ \begin{smallmatrix} -b_2 b_1 & -b_2^2 \\ b_1^2 & b_1b_2 \end{smallmatrix} \right](s)$, and picking particular entries of $V\Phi$, we deduce the expressions
    \begin{align*}
	Vb_1(t) &= V\Phi_{11}(t) = \int_0^t b_2(s) b_1(s) (b_2(t) b_1(s)-b_1(t)b_2(s)) \kappa_\perp (\varphi_s)\ ds, \\
	Vb_2(t) &= V\Phi_{12}(t) = \int_0^t b_2^2(s) (b_2(t) b_1(s)-b_1(t)b_2(s)) \kappa_\perp (\varphi_s)\ ds.
    \end{align*}
    In particular, we get 
    \begin{align*}
	V\left( \frac{1}{b_2} \right)(t) &= \frac{-Vb_2}{b_2^2}(t) = \frac{-1}{b_2^2(t)} \int_0^t b_2^2(s) (b_2(t) b_1(s)-b_1(t)b_2(s)) \kappa_\perp (\varphi_s)\ ds, \\
	V\left( \frac{b_1}{b_2} \right)(t) &= \frac{b_2 Vb_1 - b_1 Vb_2}{b_2^2}(t) = \frac{1}{b_2^2(t)} \int_0^t b_2(s) (b_2(t) b_1(s)-b_1(t)b_2(s))^2 \kappa_\perp (\varphi_s)\ ds.
    \end{align*}
    
    Notice further the following cocycle property: 
    \begin{align}
	\Phi(\x,v,t) = \Phi(\varphi_s(\x,v), t-s) \Phi(\x,v,s), \quad t\ge s,
	\label{eq:cocyclePhi}
    \end{align}
    true since, as functions of $t$ both sides satisfy the ODE $\frac{d}{dt} U + \Km (\varphi_{t}) U = 0$ with matching condition at $t = s$. This equality can be recasted as
    \begin{align*}
	\Phi(\varphi_s(\x,v), t-s) = \Phi(\x,v,t)\Phi^{-1} (\x,v,s). 	
    \end{align*}
    In particular, looking at the ${}_{(1,2)}$ entry in this matrix equality, we obtain the relation
    \begin{align}
	b_2(\varphi_s, t-s) = b_2(t) b_1(s) - b_1(t) b_2(s). 
	\label{eq:cocycle}
    \end{align}
    
    In particular, when $(M,g)$ is simple with constants $C_1, C_2$ as in \eqref{eq:simpleclaim}, we can deduce the following estimate
    \begin{align*}
	\left| V\left( \frac{1}{b_2} \right)(t) \right| = \frac{1}{b_2^2(t)}\left| \int_0^t b_2^2(s) b_2(\varphi_s, t-s)\kappa_\perp (s)\ ds \right|\le \frac{\|d\kappa\|_\infty}{C_1^2 t^2} C_2^3 \int_0^t s^2(t-s)\ ds = \frac{\|d\kappa\|_\infty C_2^3 t^2}{12 C_1^2}.
    \end{align*}
    We then obtain the exact same estimate for $V(b_1/b_2)$:  
    \begin{align*}
	\left| V\left( \frac{b_1}{b_2} \right)(t) \right| = \frac{1}{b_2^2(t)}\left| \int_0^t b_2(s) b_2^2(\varphi_s, t-s)\kappa_\perp (s)\ ds \right|\le \frac{\|d\kappa\|_\infty}{C_1^2 t^2} C_2^3 \int_0^t s(t-s)^2\ ds = \frac{\|d\kappa\|_\infty C_2^3 t^2}{12 C_1^2}.
    \end{align*}
    The lemma is proved.
\end{proof}

As a corollary of Lemma \ref{lem:Wkernelbound}, we can revisit the case without connection here. Indeed, the function $w(\x,v,t) = - V\left( \frac{b_1}{b_2} \right)(t)$ is the kernel of the error operator in the case without connection, so we record the operator estimate here. We recall that 
\begin{align*}
    Wf(\x) &= \frac{1}{2\pi} \int_{S_\x} \int_0^{\tau(\x,v)} \frac{w(\x,v,t)}{b_2(\x,v,t)} f(\varphi_t(\x,v))\ b_2(\x,v,t)\ dt\ dS(v) = \int_M {\cal W}(\x,y) f(\y)\ dM_\y,  
\end{align*}
where we have defined ${\cal W} (\x,\y) := \frac{1}{2\pi} \frac{w(\x,Exp_\x^{-1} (\y))}{b_2(\x, Exp_\x^{-1} (\y))}$. Then we write
\begin{align*}
    \|W\|_{L^2\to L^2}^2 &= \int_M \int_M {\cal W}(\x,\y)^2\ dM_\x\ dM_\y \\
    &= \frac{1}{4\pi^2} \int_M \int_{S_\x} \int_0^{\tau(\x,v)} \frac{w^2(\x,v,t)}{b_2(\x,v,t)}\ dt\ dS(v)\ dM_\x \\
    &\le \frac{1}{4\pi^2} \int_M \int_{S_\x} \int_0^{\tau(\x,v)} \left( \frac{\|d\kappa\|_\infty C_2^3 t^2}{12 C_1^2} \right)^2 \frac{1}{C_1 t}\ dt\ dS(v) \\
    &\le \frac{\Vol(M)}{2\pi} \|d\kappa\|_\infty^2 \frac{C_2^6}{144 C_1^5} \frac{\tau_\infty^4}{4}. 
\end{align*}

As a result, we obtain the bound: 
\begin{align*}
    \|W\|_{L^2\to L^2} \le \|d\kappa\|_\infty \frac{C_2^3 \tau_\infty^2}{24 C_1^{\frac{5}{2}}} \left( \frac{\Vol(M)}{2\pi} \right)^{\frac{1}{2}}.  
\end{align*}

\section{Proof of Lemma \ref{lem:tau}} 

\begin{proof}[Proof of Lemma \ref{lem:tau}] Let us define the function
    \[ G(\x,v) := b_2(\x,v,\tau(\x,v)) X_\perp \tau(\x,v) - b_1(\x,v,\tau(\x,v)) V\tau(\x,v). \]
    We will show that $G$ vanishes identically by proving that $G|_{\partial_- (SM)} =0 $ and $XG = 0$ on $SM$. Recall that $\tau$ satisfies 
    \begin{align}
	X\tau = -1 \quad (SM) , \qquad \tau|_{\partial_- (SM)} = 0.
	\label{eq:tau}
    \end{align}
    Using the structure equations, this implies 
    \begin{align*}
	X (X_\perp \tau) = [X,X_\perp] \tau = -\kappa V\tau, \qquad X(V\tau) = [X,V]\tau = X_\perp \tau. 
    \end{align*}
    From \eqref{eq:tau}, we deduce $V\tau|_{\partial_- (SM)} = 0$ and for $(\x,v)\in \partial_- (SM)$, $b_2(\x,v,\tau) = b_2(\x,v,0) = 0$, therefore $G|_{\partial_- (SM)} = 0$. On to showing $XG = 0$, we recall the definition $\Phi(\x,v,t) = \left[ \begin{smallmatrix} b_1 & b_2 \\ c_1 & c_2 \end{smallmatrix} \right]$ and set $\Psi(\x,v) := \Phi (\x,v,\tau(\x,v))$. Using \eqref{eq:cocyclePhi}, we have the property 
    \[ \Phi(\x,v,\tau(\x,v)) = \Phi(\varphi_t(\x,v), \tau(\x,v)-t) \Phi(\x,v,t), \]
    where $\tau(\x,v)-t = \tau(\varphi_t(\x,v))$. Differentiating $\frac{d}{dt}|_{t=0}$, we then obtain
    \begin{align*}
	0 = (X\Psi)(\x,v) - \Psi(\x,v) \Km(\x,v), 
    \end{align*} 
    which upon looking at the entries ${}_{(1,1)}$ and ${}_{(1,2)}$ gives
    \[ X(b_1(\x,v,\tau(\x,v))) = -\kappa(x) b_2(\x,v,\tau(\x,v)), \qquad X(b_2(\x,v,\tau(\x,v))) = b_1(\x,v,\tau(\x,v)).      \]
    In short, we obtain the matrix equation 
    \begin{align*}
	X \left[
	    \begin{array}{cc}
		X_\perp \tau & b_1(\x,v,\tau) \\ V\tau & b_2 (\x,v,\tau)
	    \end{array}
	\right] +\left[
	    \begin{array}{cc}
		0 & \kappa(\x) \\ -1 & 0
	    \end{array}
	\right] \left[
	    \begin{array}{cc}
		X_\perp \tau & b_1(\x,v,\tau) \\ V\tau & b_2 (\x,v,\tau)
	    \end{array}
	\right] = 0, \qquad (\x,v) \in SM,
    \end{align*}
    out of which $XG =0$ is just Liouville's formula.     
\end{proof}

%\newpage
%\bibliographystyle{siam}
%\bibliography{../../000-TexTouch/bibliography/bibliography}

\end{document}